\documentclass[final, a4paper, 10pt]{article}

\usepackage{amsfonts}
\usepackage{amsmath}
\usepackage{amssymb}
\usepackage{amscd}
\usepackage{amsthm}

\usepackage[mathscr]{euscript}
\usepackage{color}
\usepackage{fancyhdr}
\usepackage[a4paper,scale=0.752]{geometry}
\usepackage{hyperref}

\usepackage{bm}
\usepackage{cite}
\usepackage{showkeys}

\usepackage{accents}
\usepackage{graphicx}
\usepackage{array}

\usepackage{hyperref}

\usepackage{todonotes}

\usepackage[tikz]{mdframed}
\usepgflibrary{arrows.meta}
\usepackage{pgfplots}
\pgfplotsset{compat=1.15}
\usepgfplotslibrary{fillbetween}

\definecolor{Blue}{cmyk}{1.0,0.3,0.0,0.0}
\definecolor{Orange}{cmyk}{0.0,0.3,1.0,0.0}
\definecolor{Red}{cmyk}{0.0,1.0,1.0,0.0}
\definecolor{Green}{cmyk}{1,0,0.6,0.0,0.0}

\newtheorem{definition}{Definition}
\newtheorem{lemma}[definition]{Lemma}
\newtheorem{algorithm}[definition]{Algorithm}
\newtheorem{theorem}[definition]{Theorem}
\newtheorem{corollary}[definition]{Corollary}

\newtheorem{remark}[definition]{Remark}

\newcommand*{\N}{\ensuremath{\mathbb{N}}}
\newcommand*{\Z}{\ensuremath{\mathbb{Z}}}
\newcommand*{\R}{\ensuremath{\mathbb{R}}}
\newcommand*{\C}{\ensuremath{\mathbb{C}}}
\newcommand{\EE}{\ensuremath{\mathrm{e}}}

\renewcommand{\i}{\mathrm{i}}
\renewcommand{\phi}{\varphi}
\renewcommand{\rho}{{\varrho}}
\renewcommand{\epsilon}{{\varepsilon}}

\renewcommand{\d}[1]{\,\mathrm{d}#1 \,}

\newcommand{\E}{\mathcal{E}}

\newcommand{\J}{\mathcal{J}} 
\newcommand{\p}{{\mathrm{per}}}

\renewcommand{\k}{\underline{k}}

\newcommand{\grad}{\nabla}

\definecolor{kit-blue20}{rgb}{0.8549, 0.8784, 0.9333}

\begin{document}

\sloppy

\title{A nonuniform mesh method for wave scattered by periodic surfaces }
\author{Tilo Arens\thanks{Institute for Applied and Numerical mathematics, Karlsruhe Institute of Technology (KIT), Karlsruhe, Germany; \texttt{tilo.arens@kit.edu}} \and
Ruming Zhang\thanks{Institute for Applied and Numerical mathematics, Karlsruhe Institute of Technolog (KIT), Karlsruhe, Germany; \texttt{ruming.zhang@kit.edu}}}
\date{}

\maketitle

\begin{abstract}
In this paper, we propose a new nonuniform mesh method to simulate acoustic scattering problems in two dimensional periodic structures with non-periodic incident fields numerically. As existing methods are difficult to extend to higher dimensions, we have designed the new method with such extensions in mind. With the help of the Floquet-Bloch transform, the solution to the original scattering problem is written as an integral of a family of quasi-periodic problems. These are defined in bounded domains for each value of the Floquet parameter which varies in a bounded interval. The key step in our method is the numerical approximation of the integral by a quadrature rule adapted to the regularity of the family of quasi-periodic solutions.  We design a nonuniform mesh with a Gaussian quadrature rule applied on each subinterval. We prove that the numerical method converges exponentially with respect to both the number of subintervals and the number of Gaussian quadrature points. Some numerical experiments are provided to illustrate the results.
\end{abstract}

\section{Introduction}
  
Wave propagating in periodic structures has been a challenging and interesting topic in both theoretical analysis and numerical simulations for the past decades. For quasi-periodic incident fields, there is a  well-established framework (see \cite{Nedel1991,Dobso1992a, Abbou1993,Dobso1994,Bruno1993,Kirsc1993,Bao1994a,Bao1995,Bao1996,Bao1997,Bao2000,Arens1999,Arens2010,Stryc1998,Schmi2003,Elsch1998} for 2D and 3D acoustic, elastic and electromagnetic waves) to reduce the problem to a bounded domain. However, when the incident field is non-periodic, this approach no longer works, requiring much more sophisticated tools to tackle the problem. 

One possible approach is to apply techniques available for rough surface scattering. The basis is provided by an analysis of variational formulations for such problems in weighted Sobolev spaces \cite{Chand2005, Chand2010}. Possible numerical approaches include the integral equation method as studied in \cite{Meier2000, Chand2002, Hasel2004, Arens2006a, Lechl2008a, Li2016} for the Helmholtz equation in two dimensions and \cite{Li2011} for the full Maxwell system in three dimensions. Alternatively, the finite section method also provides a convergent algorithm to approximate the original problem by a bounded one, see \cite{Meier2001,Chand2002,Chand2010} for its applications in both boundary integral equations and finite element methods.

The principal drawback of this approach is the loss of all information and structure relying on the periodic nature of the scatterer. A powerful tool to exploit such structure is provided by the Floquet-Bloch transform, and efficient numerical methods have been developed based on this transform.

For example, penetrable periodic media are considered in \cite{Coatl2012, Hadda2015}, for scattering by periodic surfaces we refer to \cite{Lechl2015e, Lechl2016a, Lechl2017}. Note that the approach has also been extended to the three dimensional bi-periodic surfaces in \cite{Lechl2016b}. The method is proved to be convergent for 2D cases in \cite{Lechl2016a,Lechl2017}, but for 3D cases, convergence proofs are available only for some special situations in \cite{Lechl2016b}. 

After application of the Floquet-Bloch transform, the problem is reduced to solving a family of fully quasi-periodic problems for a range of the Floquet parameter. The solution of the original problem is obtained by inverting this transform. Although, numerically, this amounts to the approximate evaluation of an integral, it nevertheless proves to be a challenging task due to the presence of singularities.

Based on a detailed study of the regularity of the integrand, a high order method has been developed for 2D cases in \cite{Zhang2017e}, achieving any algebraic order of convergence. However, this scheme cannot easily be extended to the 3D case due to the complicated structure of the singularities. This has motivated the research presented in this paper, in which we design a new nonuniform mesh method for the 2D case, which is much more readily extendable to 3D cases.

Based on the singularities of the quasi-periodic solution with respect to the Floquet parameter, we first generate graded meshes in the integration interval, and then apply Gaussian quadrature rule in each sub-interval. The convergence analysis is based on bounds for analytic extensions of the solutions of the quasi-periodic problems with respect to the Floquet parameter. We prove exponential convergence with respect to both the number of graded mesh points and the number of Gaussian nodal points in each subinterval. These estimates are then coupled to error-estimates for the finite element method used to solve each quasi-periodic problem. Finally we give some numerical examples to illustrate our theoretical results.

The rest of this paper is organized as follows. The mathematical model is introduced in Section \ref{sec:model} and the Floquet-Bloch transform is reviewed in Section \ref{sec:floquet-bloch}. Then we discuss the analytic extension of quasi-periodic solutions with respect to the Floquet parameter in Section \ref{sec:extend}. In Section \ref{sec:meshes},  nonuniform meshes are designed for definite integrals with square root singularities. With the results in Section \ref{sec:extend}, we prove exponential convergence of the numerical method, and also give some numerical experiments in the last section.

\section{Mathematical model of scattering problems}
\label{sec:model}

We consider the propagation of time-harmonic waves in the two-dimensional domain $\Omega$ bounded from below by the curve $\Gamma$ given as the graph of a $2\pi$-periodic function $\zeta$, i.e.
\[
  \Omega = \Big\{ (x_1, x_2) : \, x_1 \in \R\, , \; x_2 > \zeta(x_1) \Big \} \, , \qquad 
  \Gamma = \Big\{ (x_1, \zeta(x_1)) : \, x_1\in\R \Big\} \, .
\]
The total field $u$ is assumed to satisfy the Helmholtz equation for some positive wave number $k$,
\begin{equation}
  \label{eq:sca1}
  \Delta u + k^2 \, u = 0 \qquad \text{in } \Omega
\end{equation}
and to satisfy a Dirichlet boundary condition
\begin{equation}
  u = 0 \qquad \text{on } \Gamma \, .
\end{equation}
As is usual, the total field is split into the given incident field and the unknown scattered field, $u = u^i + u^s$. A suitable radiation condition must be imposed on the scattered field to ensure uniqueness and existence of solution, and this requires some additional definitions. For detailed derivations and proofs of the statements made below we refer to \cite{Chand2005, Chand2010}.

Let $H>\max_{t\in\R}\{\zeta(t)\}$ and $\Gamma_H:=\R\times\{H\}$ be a straight horizontal line above $\Gamma$. Let $\Omega_H$ be the periodic strip between $\Gamma$ and $\Gamma_H$. For a visualization of the geometric setting we refer to Figure \ref{fig:sample}.

\begin{figure}
    \centering
    \includegraphics[width=0.8\textwidth]{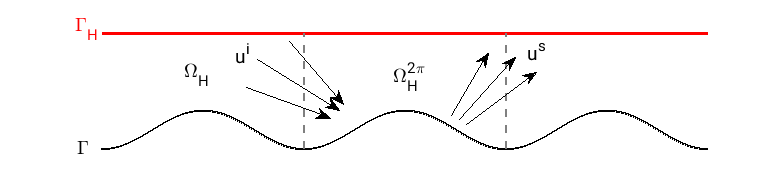}
    \caption{Periodic structure.}
    \label{fig:sample}
\end{figure}

The appropriate spaces for solutions of such a rough surface scattering problem are horizontally weighted Sobolev spaces. Let $H^s(\Omega_H)$ and $H^s_{loc}(\Omega_H)$ denote the standard Sobolev spaces with real exponent $s$.  For any $r\in\R$, let the weighted space $H^s_r(\Omega_H)$ be defined by
\[
  H^s_r(\Omega_H) := \Big\{\phi\in H^s_{loc}(\Omega_H):\, (1+x_1^2)^{r/s}\phi(x_1,x_2)\in H^s(\Omega_H)\Big\} \, .
\]
To accomodate the Dirichlet boundary condition, we also define
\[
  \widetilde{H}_r^1(\Omega_H):=\Big\{\phi\in H_r^1(\Omega_H):\,\phi\big|_{\Gamma}=0\Big\} \, .
\]

To guarantee that the scattered field $u^s$ is propagating upwards, we require that $u^s$ satisfies the following radiation condition:
\begin{equation}
  \label{eq:sar}    
  u^s(x_1,x_2) = \frac{1}{\sqrt{2\pi}}\int_\R \EE^{\i x_1\xi + \i \sqrt{k^2-\xi^2}(x_2-H)} \, \widehat{u}^s(\xi,H) \d\xi, \quad x_2>H,
\end{equation}
where the square root takes non-negative real- and imaginary parts, $\widehat{u}^s(\xi,H)$ is the Fourier transform of $u^s(\cdot,H)$. This condition is also known as the \emph{spectral amplitude representation} and formally expresses the scattered field as a linear superposition of plane upward-propagating and evanescent waves. Detailed arguments on why and in what sense the expression on the right hand side of \eqref{eq:sar} makes sense for $u^s \in H^1_r(\Omega_H)$ for all $|r| < 1$ are given in \cite{Chand2010}. Combining \eqref{eq:sar} with the Neumann trace operator on $\Gamma_H$ gives rise to the Dirichlet-to-Neumann map $T^+ : H^{1/2}_r(\Gamma_H) \to H^{-1/2}_r(\Gamma_H)$,
\[
  (T^+\phi)(x_1) = \i \int_\R \sqrt{k^2-\xi^2} \, \EE^{\i x_1\xi} \, \widehat{\phi}(\xi) \d\xi \, , \quad \text{ where } \phi(x_1) = \int_\R \EE^{\i x_1\xi} \, \widehat{\phi}(\xi) \, \d\xi \, .
\]

With these definitions, we are able to formulate the scattering problem under consideration: Given an incident field $u^i \in H^1_r(\Omega_H)$, where $|r|<1$, find $u \in \widetilde{H}^1_r(\Omega_H)$ such that $u$ satisfies the Helmholtz equation \eqref{eq:sca1} in $\Omega_H$ in the weak sense and the boundary condition
\begin{equation}
   \label{eq:sca3}
   \frac{\partial u}{\partial x_2}-T^+ u = \frac{\partial u^i}{\partial x_2}-T^+ u^i =: f \quad\text{ in } H^{-1/2}_r(\Gamma_H) \, ..
\end{equation}

Explicitly, the weak form of this scattering problem is the following variational formulation: given $f\in H_r^{-1/2}(\Gamma_H)$, where $|r|<1$, find $u\in\widetilde{H}_r^1(\Omega_H)$ such that
\begin{equation}
 \label{eq:var}
 \int_{\Omega_H}\left[\grad u\cdot\grad\overline{\phi}-k^2 u\overline{\phi}\right]\d x-\int_{\Gamma_H}T^+\left(u\big|_{\Gamma_H}\right)\overline{\phi}\d s=\int_{\Gamma_H}f\overline{\phi}\d s
\end{equation}
for all $v\in\widetilde{H}_{-r}^1(\Omega_H)$. 

\begin{theorem}[\cite{Chand2010}]
 For any $|r|<1$, given $f\in H^{-1/2}_r(\Gamma_H)$, there is a unique solution $u\in \widetilde{H}_r^1(\Omega_H)$ of the problem \eqref{eq:var}.
\end{theorem}

\begin{remark}
 In this paper, we only consider scattering problems with Dirichlet boundary condition on $\Gamma$. However, the method can also extended to other boundary conditions (e.g., impedance boundary conditions) or penetrable inhomogeneous media. 
\end{remark}

\section{Floquet-Bloch transformed field and the regularity}
\label{sec:floquet-bloch}

In this section, we recall the definition and some properties of the Floquet-Bloch transform, and apply it to the solution of the original problem \eqref{eq:var}. Finally, we introduce the regularity result of the transformed field. For details we refer to \cite{Lechl2016,Zhang2017e}.

\subsection{The Floquet-Bloch transform}

We define the Floquet-Bloch transform of $\phi\in C_0^\infty(\Omega_H)$ as
\[
(\J\phi)(\alpha,x)=\sum_{j\in\Z}\phi(x_1+2\pi j,x_2)e^{-\i\alpha(x_1+ 2\pi j)},
\]
where $\alpha\in (-1/2,1/2]$ (called Floquet parameter) and $x\in\Omega_H^{2\pi}:=\Omega_H\cap [-\pi,\pi]\times\R$ (see Figure \ref{fig:sample}). As $\phi$ has a compact support, the transform is well-defined for $\alpha\in (-1/2,1/2]$ and $x\in\Omega_H^{2\pi}$. It is also easy to check that when $\alpha$ is fixed, $(\J\phi)(\alpha,\cdot)$ is $2\pi$-periodic in $x_1$, i.e.,
\[(\J\phi)\left(\alpha,\left(\begin{matrix}
x_1+2\pi \\ x_2                             
                             \end{matrix}\right)\right)
=(\J\phi)(\alpha,x).
\]
Moreover, $e^{\i\alpha x_1}w(\alpha,x)$ is $1$-periodic in $\alpha$ for fixed $x$.

To introduce properties of the Bloch transform, we define the space $H^m\left((-1/2,1/2];H^s(\Omega_H^{2\pi})\right)$ ($m\in\N$) equipped with the norm:
\[
\|\phi\|_{H^m\left((-1/2,1/2];H^s(\Omega_H^{2\pi})\right)}:=\left[\sum_{\ell=0}^m\int_{-1/2}^{1/2} \left\|\partial^\ell_\alpha\phi(\alpha,\cdot)\right\|^2_{H^s(\Omega^{2\pi}_H)}\right]^2.
\]
This definition is extended to any real number $r$ by interpolation and duality arguments. We also define the subspace $H^r\left((-1/2,1/2];H^s_\p(\Omega^\Lambda_H)\right)$ that contains functions which are $2\pi$-periodic with respect to $x_1$ with fixed  $\alpha$. We conclude our overview of the properties of the Floquet-Bloch transform with the following theorem.

\begin{theorem}\label{th:Bloch}
 The transform $\J$ is extended to an isomorphism between $H^s_r(\Omega_H)$ and $H^r\left((-1/2,1/2];H^s_\p(\Omega^\Lambda_H)\right)$ for any $s,\,r\in\R$.
 \[\left(\J^{-1} w\right)(x)=\int_{-1/2}^{1/2} w(\alpha,x)e^{\i \alpha x_1}\d\alpha,\quad x\in\Omega_H.
 \] 
 When $s=r=0$, $\J$ is an isometry with its inverse and $\J^{-1}=\J^*$.
\end{theorem}

\subsection{Bloch transformed field and regularity}

Following the process in \cite{Lechl2016}, we apply the Floquet-Bloch transform to the total field $u$, then $w(\alpha,x):=(\J u)(\alpha,x)$ satisfies the following variational equation with the test function $\psi(\alpha,x)=(\J\phi)(\alpha,x)$:
\begin{equation}
 \label{eq:var_bloch}
 \int_{-1/2}^{1/2} a_\alpha(w(\alpha,\cdot),\psi(\alpha,\cdot))\d\alpha=\int_{-1/2}^{1/2}\int_{\Gamma_H^{2\pi}}F(\alpha,\cdot)\overline{\psi}(\alpha,\cdot)\d s\d\alpha.
\end{equation}
where
\begin{eqnarray*}
 && a_\alpha(\xi,\eta)=\int_{\Omega^{2\pi}_H}\left[\grad \xi\cdot\grad\overline{\eta}-2\i\alpha\frac{\partial\xi}{\partial x_1}\overline{\eta}+(\alpha^2-k^2) \xi\overline{\eta}\right]\d x-\int_{\Gamma^{2\pi}_H}T^+_\alpha\left[\xi\big|_{\Gamma^{2\pi}_H}\right]\overline{\eta}\d s \, , \\
 && F(\alpha,x)=(\J f)(\alpha,x) \, .
\end{eqnarray*}
and $T^+_\alpha$ is the periodic Dirichlet-to-Neumann map with index $\alpha$ from $H^{1/2}_\p(\Gamma_H^{2\pi})$ to $H^{-1/2}_\p(\Gamma_H^{2\pi})$. It has the following form:
\[(T^+_\alpha\phi)(x_1)=\i\sum_{j\in\Z}\sqrt{k^2-(\alpha+j)^2}\widehat{\phi}(j)e^{\i j x_1} \qquad\text{where}\quad \phi(x_1)=\sum_{j\in\Z}\widehat{\phi}(j)e^{\i j x_1}.\]

From the equivalence between \eqref{eq:var} and \eqref{eq:var_bloch}, we have the following results. For proofs we refer to \cite{Lechl2016,Lechl2016a,Lechl2017}.

\begin{theorem}
 Given $f\in H_r^{-1/2}(\Gamma_H)$ for $|r|<1$, the variational problem \eqref{eq:var_bloch} is uniquely solvable in $H^r\left((-1/2,1/2];\widetilde{H}^1_\p(\Omega_H^{2\pi})\right)$. Moreover,
 \begin{enumerate}
  \item when $f\in H_r^{1/2}(\Gamma_H)$ and $\zeta\in C^{2,1}$, $w\in H^r\left((-1/2,1/2];\widetilde{H}^2_\p(\Omega^{2\pi}_H)\right)$ and $u\in H^2_r(\Omega_H)$;
  \item when $f\in H_r^{-1/2}(\Gamma_H)$ for $r\in (1/2,1)$, then $w\in L^2\left((-1/2,1/2];\widetilde{H}^1_\p(\Omega_H^{2\pi})\right)$ equivalently satisfies
  \begin{equation}
   \label{eq:var_single}
   a_\alpha(w(\alpha,\cdot),\phi)=\int_{\Gamma^{2\pi}_H}F(\alpha,\cdot)\overline{\phi}\d s
  \end{equation}
for any $\alpha\in (-1/2,1/2]$ and $\phi\in\widetilde{H}^1_\p(\Omega^{2\pi}_H)$.
 \end{enumerate}
\end{theorem}

{In particular for numerical applications, it is important to have a more detailed understanding of the regularity properties of $w(\alpha,x)$ with respect to the Floquet parameter $\alpha$.} Before the study of these properties, we first introduce the following notations and spaces. For any fixed positive wavenumber $k$, let
\[\underline{k}:=\min\{|n-k|:\,n\in\Z\}.\]
From $1$-periodicity of $\EE^{\i \alpha x_1}w(\alpha,x)$ with respect to $\alpha$, the inverse Bloch transform (see Theorem \ref{th:Bloch}) is written equivalently as
 \[\left(\J^{-1} w\right)(x)=\int_{-\underline{k}}^{1-\underline{k}} w(\alpha,x) \, \EE^{\i \alpha x_1} \, \d\alpha,\quad x\in\Omega_H.
 \] 
Let 
\[
  E:=\left\{\alpha\in[-\underline{k},1-\underline{k}]:\,|n-\alpha|=k\text{ for some }n\in\Z\right\} .
\]
Then. from direct calculation,
\[
 E=\begin{cases}
    \{-\underline{k},\,1-\underline{k}\},\quad\text{ when }k=n/2\text{ for some }n\in\N_+;\\
    \{-\underline{k},\,\underline{k},\,1-\underline{k}\},\quad\text{ otherwise.}
   \end{cases}
\]
This implies that $E$ contains the edge of $[-\underline{k},1-\underline{k}]$, and may also contain a point in the interior this interval.

The formulation of the regularity properties of the Bloch transformed field requires some appropriate function spaces. Let $\mathcal{I}\subset\R$ denote a bounded open interval, ${D}\subset\R^2$ a bounded domain, and $S({D})$ a Sobolev space independent of $\alpha$ of functions defined in ${D}$. First, define a space of functions that depend analytically on $\alpha$:
\[
\begin{aligned}
 C^\omega(\mathcal{I},S({D})):=&\Bigg\{\phi\in\mathcal{D}'(\mathcal{I}\times{D})):\,\forall\alpha_0\in\mathcal{I},\,\exists\,\delta>0,\,s.t.,\,\forall\alpha\in(\alpha_0-\delta,\alpha_0+\delta)\cap\mathcal{I},\Bigg.\\
& \left.\,\exists C>0,\,\phi_n\in S(D),\,s.t., \phi(\alpha,x)=\sum_{n=0}^\infty(\alpha-\alpha_0)^n \phi_n(x),\,\|\phi_n\|_{S(D)}\leq C^n\right\}.
\end{aligned}
\]
Also, let the subspace of functions that are $C^n$-continuous with respect to $\alpha$ be defined as:
\[
 \begin{aligned}
  C^n(\mathcal{I},S(D)):=&\Bigg\{\phi\in\mathcal{D}'(\mathcal{I}\times D):\,\forall\alpha\in\mathcal{I},\,j=0,1,\dots,n,\,\frac{\partial^j f(\alpha,\cdot)}{\partial\alpha^j}\in S(D),\Bigg.\\
  &\left.\text{moreover},\,\left\|\frac{\partial^j f(\alpha,\cdot)}{\partial \alpha^j}\right\|_{S(D)}\text{ is uniformly bounded for }\alpha\in\mathcal{I}\right\}.
 \end{aligned}
\]
The regularity of the Bloch transformed field can be characterized through two properties that we here formulate for a function $\phi\in C^0(\mathcal{I};S(D))$:
\begin{enumerate}
 \item For any subinterval $\mathcal{I}_0\subset\mathcal{I}\setminus E$, $\phi\in C^\omega(\mathcal{I}_0;S(D))$.
 \item For any $\alpha_0\in\mathcal{I}\cap E$, there is a sufficiently small $\delta>0$ and a pair $\phi_1,\phi_2\in C^\omega(\mathcal{I}_0;S(D))$ such that
 \[\phi(\alpha,\cdot)=\phi_1(\alpha,\cdot)+\sqrt{\alpha-\alpha_0}\,\phi_2(\alpha,\cdot),\]
 where $\mathcal{I}_0=(\alpha_0-\delta,\alpha_0+\delta)\cap\mathcal{I}$.
\end{enumerate}
The space of functions that satisfies both these properties will be denoted as 
\[\mathcal{A}^\omega(\mathcal{I};S(D);E):=\left\{\phi\in C^0(\mathcal{I};S(D)):\, \phi \text{ satisfies condition 1 and 2}\right\}.\]

With the help of all these definitions, the regularity of the Bloch transformed  field $w(\alpha,x)$ can now be stated in the following theorem. For a proof we refer to Theorem 16 in \cite{Zhang2017e}.

\begin{theorem}
 \label{th:reg_bloch}
 Given any $f\in H^{-1/2}_r(\Gamma_H)$  such that $\J f=F\in\mathcal{A}^\omega\left((-\underline{k},1-\underline{k}];H^{-1/2}_\p(\Gamma_H^{2\pi});E\right)$ where $r\in(1/2,1)$, then the transformed solution $\J u=w\in \mathcal{A}^\omega\left((-\underline{k},1-\underline{k}];\widetilde{H}^{1}_\p(\Omega_H^{2\pi});E\right)$. Moreover, if $f\in H^{1/2}_r(\Gamma_H)$ and $\zeta\in C^{2,1}(\R)$, $w\in \mathcal{A}^\omega\left((-\underline{k},1-\underline{k}];\widetilde{H}^{2}_\p(\Omega_H^{2\pi});E\right)$.
\end{theorem}

From Theorem \ref{th:reg_bloch}, the transformed field $w(\alpha,\cdot)$ has  only a finite number of square-root singularities. As $w(\alpha,\cdot)$ can be readily computed by well-established methods, the only difficulty is to approximate the inverse Bloch transform. Note that although in \cite{Zhang2017e}, one of the authors has proposed a highly efficient numerical method for the approximation, it is extremely difficult to extend this approach to scattering problems with bi-periodic structures in three dimensional space. With the ultimate goal of such an extension in mind, we introduce a nonuniform mesh method for the numerical approximation below. The estimation of quadrature errors arising in this method relies on bounds of analytic extensions of $\varphi_1$ and $\varphi_2$ in property 2 in the complex plane with respect to $\alpha$. Thus, before the introduction to the adaptive method, we have to investigate the extension of the solutions to a neighbourhood of $(-\underline{k},1-\underline{k}]$ in $\C$.

\section{Extension to complex quasi-periodicities}
\label{sec:extend}

The convergence analysis of our method of inversion for the Floquet-Bloch transform relies on estimates for analytic extensions with respect to $\alpha$ of the solution $w(\alpha, \cdot)$ of the quasi-periodic problem. It is the goal of this section to characterize complex neighborhoods of the real axis to which $w(\alpha,\cdot)$ may be extended analytically and to provide estimates for these extensions. Our approach is, first, to precisely define analytic extensions of the variational formulation and of the corresponding operators and, secondly, to estimate the difference between these operators and their counterparts for real $\alpha$ for small deviations from the real axis. From these results, standard perturbation theory will yield analyticity of the solution $w(\alpha,\cdot)$ as well as the required bounds.

To simplify our estimates, let us slightly modify the spaces. Denote by $V_H^\p$ the space $\widetilde{H}^{1}_\p(\Omega_H^{2\pi})$ with the norm replaced by
\[
\|\phi\|_{V_H^\p}:=\left[\int_{\Omega^{2\pi}_H}\left[\left|\grad\phi\right|^2+k^2|\phi|^2\right]\d x\right]^{1/2}.
\]
Let the norm in  $H^s_\p(\Gamma^{2\pi}_H)$ be defined by
\[
\|\phi\|_{H^s_\p(\Gamma^{2\pi}_H)}=\left[\sum_{j\in\Z}(k^2+j^2)^s\big|\widehat{\phi}_j\big|^2\right]^{1/2}.
\]
As the sesquilinear form $a_\alpha(\cdot,\cdot)$ is well defined in $V_H^\p\times V_H^\p$, from Riesz's Lemma, there is a $\mathcal{B}_\alpha:\,V_H^\p\rightarrow \left(V_H^\p\right)^*$ such that
\[
a_\alpha(\phi,\psi)=\left<\mathcal{B}_\alpha\phi,\psi\right>,
\]
where $\phi,\psi\in V_H^\p$ and $\left<\cdot,\cdot\right>$ denotes the extension of the $L^2(\Omega^{2\pi}_H)$ inner product to the $\left(V_H^\p\right)^*$-$V_H^\p$ duality.
Moreover, let $F(\alpha,\cdot) \in H^{-1/2}_\p(\Gamma^{2\pi}_H)$.

Let $\delta$ denote a small complex number, and $\mathcal{D}_{\delta,\alpha}$ the perturbation of $\mathcal{B}_\alpha$ obtained from replacing $\alpha$ by $\alpha + \i \delta$, i.e.
\begin{equation}
 \label{eq:D_plus_B}
 \left\langle \left( \mathcal{B}_\alpha + \mathcal{D}_{\delta,\alpha} \right) v, \phi \right\rangle 
 = \int_{\Omega^{2\pi}_H} \left[ \grad v \cdot \grad \overline{\phi} - 2 \i \, ( \alpha + \i \delta) \, \frac{\partial v}{\partial x_1} \, \overline{\phi} + ( (\alpha + \i  \delta)^2 - k^2) \, v \, \overline{\phi} \right] \d x - \int_{\Gamma^{2\pi}_H} T^+_{\alpha+\i\delta} v \, \overline{\phi} \, \d x \, .
\end{equation}
For the moment, we consider $T^+_{\alpha+\i\delta}$ as a formal symbol only, a precise definition of this operator will be given below. A direct calculation shows
\begin{equation}
 \label{eq:D_delta_alpha}
 \left\langle \mathcal{D}_{\delta,\alpha} v , \phi \right\rangle
 = \int_{\Omega^{2\pi}_H} \left[ 2 \delta \, \frac{\partial v}{\partial x_1} \,  \overline{\phi} + (2\i\alpha\delta - \delta^2) \, v \, \overline{\phi} \right] \d x - \int_{\Gamma^{2\pi}_H} \left[ T^+_{\alpha+\i\delta} - T^+_{\alpha} \right] v \, \overline{\phi} \, \d s \, .
\end{equation}
Obviously, the first integral depends analytically on $\delta$ in all of $\C$. Thus we only need to investigate the analytic extension of the operator $T^+_\alpha$ which involves a countable number of functions with square-root singularities. As ${T}^+_\alpha$ is real analytic in $(-\underline{k},1-\underline{k}]$ except for the finite set $E$, we extend the operator also analytically in the neighbourhood of $(-\underline{k},1-\underline{k}]$ except for a finite number of vertical lines $\{\alpha_0\} \times \R$, where $\alpha_0 \in E$. 

To this end, we redefine the square root operator ``$\sqrt{\quad}$'' as follows.

\begin{definition}
\label{def:sr}
 For any $z\in\C\setminus\{0\}$, there is a unique representation such that
 \[
  z=r \EE^{\i\theta}, \quad r = |z| > 0 \, , \quad \theta \in \left( -\frac{\pi}{2}, \frac{3\pi}{2} \right] .
 \]
Define $\sqrt{z} = \sqrt{r} \EE^{\i\theta/2}$, where $\sqrt{r}$ denotes the usual square root for a positive real number. Moreover, when $z=0$, $\sqrt{z}=0$.
\end{definition}

Via Definition \ref{def:sr}, the square root function is analytically extended to the complex plane except for the  negative imaginary axis. Thus for each term in the formal expression for $T^+_{\alpha+\i\delta}$, the map $\delta \mapsto \sqrt{k^2 - (\alpha + j + \i \delta)^2}$ is real analytic in $\R$  when $|\alpha+j|\neq k$.
 
Consider now the analytic extension of the terms in the definition of the operator ${T}^+_\alpha$. Let
\[
  A_1=(-\k,\k), \qquad A_2=(\k,1-\k) \, .
\]
Note that if $\k=0$, $A_1=\emptyset$ while when $\k=1/2$, $A_2=\emptyset$; while neither is empty otherwise.  The observations above show that for each $j \in \Z$, the function $\delta \mapsto  \sqrt{ k^2 - (\alpha + j + \i\delta )^2}$ is analytic in the strips $A_m +\i \R$, $m=1$, $2$. We will show in Theorem \ref{th:pert_dtn} below that the series over all these terms indeed converges and bound its difference from $T^+_\alpha$. Before we can establish this result, however, we require two technical estimates for these square roots terms.

\begin{lemma}
 For any $\alpha \in A_m$, $m = 1$, $2$, $\delta \in \R$ and $j \in \Z$,
 \begin{equation}
  \label{eq:inequ_2}
  \left| \sqrt{ (\alpha + j + \i\delta)^2 - k^2 } - \sqrt{ (\alpha + j)^2 - k^2 } \right|
  \leq \frac{ |\delta| }{2 \left| \sqrt{k^2 - (\alpha+j)^2 } \right| } \left( \max\{ k, |\alpha+j| \} + \frac{ |\delta|^2 }{ 4 \left| \alpha + j - k \right| } \right) .
 \end{equation}
\end{lemma}

\begin{proof}
  From Definition \ref{def:sr}, we have for $x \in \R \setminus \{ 0 \}$ and $y \in \R$ that
  \[
    \left| \sqrt{x + \i y} + \sqrt{x} \right|
    = | \sqrt{x} | \, \left| 1 + \sqrt{ 1 + \i \, \frac{y}{x} } \right|
    \geq 2 \, | \sqrt{x} | \, .
  \]
  Hence, by an elementary calculation, we obtain for $\alpha \in A_m$ that
  \begin{align*}
      \left| \sqrt{ (\alpha + j + \i \delta)^2 - k^2}\right. & \left.{} - \sqrt{ (\alpha+j)^2 - k^2} \right| \\[1ex]
      & \leq \left| \sqrt{ \alpha + j + \i \delta + k } - \sqrt{ \alpha + j + k } \right| \left| \sqrt{ \alpha + j + \i \delta - k } \right| \\
      & \quad {} + \left| \sqrt{ \alpha + j + k } \right| \left| \sqrt{ \alpha + j + \i\delta - k } - \sqrt{ \alpha + j - k } \right| \\[1ex]
      & = \frac{ |\delta| \left| \sqrt{ \alpha + j + \i\delta - k } \right| }{ \left| \sqrt{ \alpha + j + \i\delta + k } + \sqrt{ \alpha + j + k } \right| } + \frac{ |\delta| |\sqrt{ \alpha + j + k }| }{ \left| \sqrt{ \alpha + j + \i\delta - k } + \sqrt{ \alpha + j - k } \right|} \\[2ex]
      & \leq \frac{ |\delta| }{ 2 } \left( \frac{ | \sqrt{ \alpha + j + \i\delta - k } | }{ | \sqrt{ \alpha + j + k } | } + \frac{ | \sqrt{ \alpha + j + k } | }{ | \sqrt{ \alpha + j - k } | } \right). 
   \end{align*}
  We further estimate, again for $x \in \R \setminus \{ 0 \}$ and $y \in \R$, 
  \[
    \left| \sqrt{ x + \i y} \right|
    = ( x^2 + y^2 )^{1/4}
    \leq | x |^{1/2} \left( 1 + \frac{y^2}{4 x^2} \right)
  \]
  Thus
  \[
    \left| \sqrt{ \alpha + j + \i\delta - k }  \right|
    \leq  \left| \sqrt{ \alpha + j - k }  \right| + \frac{\delta^2}{4 \, | \alpha + j - k |^{3/2} } \, ,
  \]
  and the assertion follows.
\end{proof}

In the next step we further estimate the leading factor in \eqref{eq:inequ_2} by estimating from below $\left| \sqrt{ k^2 - (\alpha+j)^2 } \right|$ for $j \in \Z$, when $\alpha$ is fixed.

\begin{lemma}
 \label{lemma:sqrt_estimate}
 Let $\alpha \in A_m = (a_0,a_1)$ for $m=1$, $2$. Then 
 \[
   \min_{j \in \Z} \left| \sqrt{ k^2 - (\alpha+j)^2 } \right|
   \geq \sigma \, \min \left\{ \sqrt{ \alpha - a_0 } \, , \, \sqrt{ a_1 - \alpha } \right\}
 \]
 with $\sigma = 1$ if $k > 1/2$ and $\sigma = \sqrt{\k} > 0$ if $k \leq 1/2$.
\end{lemma}

\begin{proof}
 Note first, that all factors occurring on the right hand side of the asserted lower bound are less than or equal to 1.
 
 Suppose that $k = \hat{j} + \k$, $\hat{j} \in \Z_{\geq 0}$ and write $j = \hat{j} + n$, $n \in \Z$. Then
 \begin{align*}
     \left| k^2 - ( \alpha + j )^2 \right|
      = \left| ( \hat{j} + \k )^2 - ( \alpha + \hat{j} + n )^2 \right|
     = \big| \k - \alpha - n \big| \, \big| \k + \alpha + n + 2 \hat{j} \big| .
 \end{align*}
 First, let $\alpha \in A_1 = (-\k, \k)$. Then, 
 \begin{align*}
   | \k - \alpha - n | & \geq \min \{ | \k - \alpha | \, , | \k - \alpha - 1 |
    \geq \min \{ | \k - \alpha | \, , | \k + \alpha | \} \, , \\
   | \k + \alpha + n + 2 \hat{j} | & \geq \min \{ | \k +  \alpha | \, , | \k + \alpha - 1 | \}
   \geq \min \{ | \k +  \alpha | \, , | \k - \alpha  | \} \, ,
 \end{align*}
 as well as
 \[
   | \k - \alpha - n | \geq 1 \, , \quad n \in \Z \setminus \{ 0, \, 1 \} \, ,
   \qquad \text{and} \qquad
   | ( \k + \alpha + n + 2 \hat{j} ) | \geq 1 \, , \quad n + 2 \hat{j} \in \Z \setminus \{ -1, \, 0 \} \, .
 \]
 This proves the assertion for $\alpha \in A_1$ unless $n = \hat{j} = 0$. In this case we can obviously estimate
 \[
   \left| \k - \alpha \right| \, \left| \k + \alpha \right| 
   \geq \k \, \min \{ | \k - \alpha | \, , \, | \k + \alpha | \} \, .
 \]
 Taking the square root gives the estimate for $\alpha \in A_1$.
 
 Now let $\alpha \in A_2 = (\k, 1 - \k)$. In this case, 
 \begin{align*}
   | \k - \alpha - n | & \geq \min \{ | \k - \alpha | \, , | \k - \alpha + 1 | \} 
   \geq \min \{ | \k - \alpha | \, , | 1 - \k - \alpha | \} \, , \\
   | \k + \alpha + n + 2 \hat{j} | & \geq \min \{ | \k +  \alpha - 1 | \, , | \k + \alpha | \} 
   \geq \min \{ | \k +  \alpha - 1 | \, , | \k - \alpha | \} \, ,
 \end{align*}
 as well as
 \[
   | \k - \alpha - n | \geq 1 \, , \quad n \in \Z \setminus \{ -1, \, 0 \} \, ,
   \qquad \text{and} \qquad
   | ( \k + \alpha + n + 2 \hat{j} ) | \geq 1 \, , \quad n + 2 \hat{j} \in \Z \setminus \{ -1, \, 0 \} \, .
 \]
 The assertion is proven for $\alpha \in A_2$ unless $\hat{j} = 0$ and $n \in \{ -1 \, , \, 0 \}$. In these exceptional cases we have
 \[
   \left| \k - \alpha \right| \, \left| \k + \alpha \right| 
   \geq 2 \k \, | \k - \alpha | \, , \qquad
   \left| \k - \alpha + 1 \right| \, \left| \k + \alpha - 1\right| 
   \geq 2 \k \, | 1 - \k - \alpha | \, .
 \]
 All arguments are repeated with obvious sign changes for $k = \hat{j} - \k$
\end{proof}

With the previous two lemmas, we are now able to prove boundedness of $T^+_{\alpha + \i \delta}$ and bound the difference of the two DtN maps.

\begin{theorem}
\label{th:pert_dtn}
 Let $\alpha \in A_m = (a_0,a_1)$, $m=1$, $2$. Let $\rho = | \delta | / \min \left\{ \sqrt{ \alpha - a_0 } \, , \, \sqrt{ a_1 - \alpha } \right\}$. Then $T^+_{\alpha + \i\delta} : H^{1/2}_\p(\Gamma^{2\pi}_H) \to H^{-1/2}_\p(\Gamma^{2\pi}_H)$ is bounded for any $\delta \in \R$. Moreover, its difference from $T^+_\alpha$ is bounded by
 \[
  \left\| {T}^+_{\alpha + \i\delta} - {T}^+_\alpha \right\|
  \leq \frac{ \rho }{ \sigma } + \frac{ \rho^3}{ 8 \sigma k } \, ,
 \]
 where $\sigma$ is the constant defined in Lemma \ref{lemma:sqrt_estimate}.
\end{theorem}

\begin{proof}
  Consider $\phi \in C^\infty_\p(\Gamma^{2\pi}_H)$ and its Fourier series, $\phi(x_1) = \sum_{j\in\Z} \widehat{\phi}_j \EE^{\i j x_1}$. Then
 \[
   \left({T}^+_{\alpha + \i\delta} - {T}^+_\alpha \right) \phi 
   = \i \sum_{j\in\Z} \left[ \sqrt{ k^2 - (\alpha + j + \i\delta)^2 } - \sqrt{k^2 - (\alpha + j)^2} \right] \widehat{\phi}_j \EE^{\i (\alpha+j) x_1},
 \]
 with convergence of the series ensured by the smoothness of $\phi$. Hence
 \[
   \left\| \left({T}^+_{\alpha + \i\delta} - {T}^+_\alpha \right) \phi \right\|^2_{H^{-1/2}_\p(\Gamma^{2\pi}_H)}
   = \sum_{j \in \Z} ( k^2 + |j|^2 )^{-1/2} \left| \sqrt{k^2 - (\alpha + j + \i\delta)^2} - \sqrt{k^2 - (\alpha+j)^2} \right|^2 \left| \widehat{\phi}_j \right|^2 .
 \]
 Thus, by the previous two lemmas,
 \begin{align*}
   \left\| \left({T}^+_{\alpha + \i\delta} - {T}^+_\alpha \right) \phi \right\|_{H^{-1/2}_\p(\Gamma^{2\pi}_H)}
   & \leq \sup_{j\in\Z} \frac{ \left| \sqrt{k^2 - (\alpha + j +\i\delta)^2} - \sqrt{k^2 - (\alpha+j)^2} \right| }{ \sqrt{k^2+|j|^2} } \, \| \phi \|_{H^{1/2}_\p(\Gamma^{2\pi}_H)} \\
   & \leq \sup_{j\in\Z}  \frac{ |\delta| \, \left( \max\{ k, |\alpha+j| \} + \frac{ |\delta|^2 }{ 4 \left| \alpha + j - k \right| } \right) }{2 \sigma \, \min \left\{ \sqrt{|\alpha-a_0|} \, , \, \sqrt{|\alpha-a_1|}  \right\} } \, \frac{1}{ \sqrt{ k^2 + |j|^2 } } \, \| \phi \|_{H^{1/2}_\p(\Gamma^{2\pi}_H)}   \\
   & \leq \frac{ | \delta | \left( 1 + \delta^2 \, \sup\limits_{j \in \Z} \frac{1}{8 \, | \alpha + j - k | \, \sqrt{ k^2 + | j |^2 } }\right) }{ \sigma \, \min \left\{ \sqrt{|\alpha-a_0|} \, , \, \sqrt{|\alpha-a_1|} \right\} } \, \| \phi \|_{H^{1/2}_\p(\Gamma^{2\pi}_H)} \, .
 \end{align*}
The proof is finished by observing $| \alpha + j - k | \geq \min \left\{  \alpha - a_0  \, , \,  a_1 - \alpha \right\}$.
\end{proof}

Our goal is to prove that the operator $\mathcal{B}_\alpha + \mathcal{D}_{\delta,\alpha}$ defined in \eqref{eq:D_plus_B} is boundedly invertible. To this end, for fixed $\alpha$ and $\delta$ \emph{real,} we bound the operator $\mathcal{D}_{\delta,\alpha}$ with respect to $\delta$.  Theorem \ref{th:pert_dtn} provides an estimate for the second term in \eqref{eq:D_delta_alpha}. It is also easily checked that
\begin{equation}
 \label{eq:D_domain_bound}
 \left| \int_{\Omega^{2\pi}_H} \left[ 2\delta \frac{\partial v}{\partial x_1} \overline{\phi} + (2\i\alpha\delta- \delta^2) v \overline{\phi} \right] \d x \right|
 \leq \left( 4|\delta| +\delta^2/k \right) \|v\|_{V_H^\p} \|\phi\|_{V_H^\p} \, .
\end{equation}
As we are looking at small perturbations $\delta$, we will assume that $|\delta| \leq k$. Then, as a conclusion from \eqref{eq:D_domain_bound} and Theorem \ref{th:pert_dtn}, the norm of the operator $\mathcal{D}_{\delta,\alpha}$ is bounded by:
\[
 \| \mathcal{D}_{\delta,\alpha} \|
 \leq 5 \, | \delta | + \frac{ \rho }{ \sigma } + \frac{ \rho^3}{ 8 \sigma k } \, .
\]
In \cite{Chand2005}, explicit bounds for the inverse of the unperturbed operator $\mathcal{B}_\alpha$ are provided. It is shown in Theorem 4.1 of that reference that there exists a constant $M \geq 1$ with
\[
  \left\| \mathcal{B}_\alpha^{-1} \right \|\leq M \qquad \text{for all } \alpha \in A_m \, .
\]
Standard operator perturbation results hence show that $\mathcal{B}_\alpha + \mathcal{D}_{\delta, \alpha}$ is boundedly invertible when
\begin{equation}
  \label{eq:bound_for_delta}
  5 \, | \delta | + \frac{ \rho }{ \sigma } + \frac{ \rho^3}{ 8 \sigma k }
  < \frac{1}{M} \, .
\end{equation}
In the following theorem, we provide sufficient conditions on $\delta$ to satisfy this inequality.

\begin{theorem}
 \label{theo:A_m_2_parabolas}
 Let $\alpha \in A_m = (a_0, a_1)$, $m = 1$, $2$. There exists a constant $\mathcal{C} > 0$ such that if
 \[
   | \delta | \leq \mathcal{C} \, \min \left\{ \sqrt{ \alpha - a_0 } \, , \, \sqrt{ a_1 - \alpha } \right\} \, ,
 \]
 then $| \delta | \leq k$ and the bound \eqref{eq:bound_for_delta} is satisfied. Hence, the operator $(\mathcal{B}_\alpha + \mathcal{D}_{\delta,\alpha})^{-1}$ and consequently also $w(\alpha + \i \delta, \cdot)$ depend analytically on $\alpha + \i\delta$ on the set
 \[
    \mathbf{A}_m = \big\{ \alpha + \i \delta : \alpha \in A_m \, , \; \delta \in \R \text{ with } | \delta | \leq \mathcal{C} \, \min \left\{ \sqrt{ \alpha - a_0 } \, , \, \sqrt{ a_1 - \alpha } \right\}  \big\}  \, .
 \]
\end{theorem}


\begin{proof}
 Let $\mu = \max\limits_{\alpha \in A_m} \min \left\{ \sqrt{ \alpha - a_0 } \, , \, \sqrt{ a_1 - \alpha } \right\}$ and choose
 \[
    \mathcal{C} < \min \left\lbrace \frac{k}{\mu} , \frac{1}{M \, ( 5 + \frac{1}{\sigma} + \frac{1}{8 \sigma k} ) } \right \rbrace \, . 
 \]
 Let $| \delta | \leq \mathcal{C} \, \min \left\{ \sqrt{ \alpha - a_0 } \, , \, \sqrt{ a_1 - \alpha } \right\}$. Then $| \delta | \leq k$ and
 \[
   \rho = \frac{ | \delta | }{ \min \left\{ \sqrt{ \alpha - a_0 } \, , \, \sqrt{ a_1 - \alpha } \right\} } 
   < \frac{1}{M \, ( 5 + \frac{1}{\sigma} + \frac{1}{8 \sigma k} ) }
   \leq \frac{1}{M \, ( 5 \, \min \left\{ \sqrt{ \alpha - a_0 } \, , \, \sqrt{ a_1 - \alpha } \right\} + \frac{1}{\sigma} + \frac{1}{8 \sigma k} ) }
 \]
 as $\mu \leq 1$. Note also $\rho \leq 1$ as $M \geq 1$ and $\sigma \leq 1$. Now we conclude
 \[
   5 \, | \delta | + \frac{ \rho }{ \sigma } + \frac{ \rho^3}{ 8 \sigma k }
   = \left( 5 \, \frac{| \delta |}{\rho} \, + \frac{1}{\sigma} + \frac{\rho^2}{8 \sigma k} \right) \rho
   < \frac{1}{M} \, .
 \]
 Hence \eqref{eq:bound_for_delta} is satisfied.
\end{proof}


In this case, the set $\mathbf{A}_m$ is the intersection of the interior of two parabolas in the complex plane. For  $A_m:=(a_0,a_1)$, we can write
\[
   \mathbf{A}_m = \{ \alpha + \i \delta : \alpha \geq a_0 + \delta^2/\mathcal{C}^2 \text{ and } \alpha \leq a_1 - \delta^2 / \mathcal{C}^2 \} \setminus \{ a_0, \, a_1 \}
\]
with the constant $C$ from Theorem \ref{theo:A_m_2_parabolas}
In Section \ref{sec:meshes} below, we will require analytical extensions of  $\alpha \mapsto w(\alpha,\cdot)$ to certain ellipses. Hence we will now consider ellipses contained in $\mathbf{A}_m$.

Let us recall some basic definitions and properties of an ellipses. An ellipse with center at $(s,t) \in \R^2$ and half axes $a\geq b>0$ parallel to the coordinate axes is defined as the set
\[
  \mathcal{E} = \left\{ x \in \R^2 : \frac{(x_1-s)^2}{a^2}+\frac{(x_2-t)^2}{b^2} \leq 1 \right\},
\]
The number $c:=\sqrt{a^2-b^2}$ is the called the linear eccentricity,  $(-c+s,t)$ and $(c+s,t)$ are the foci and $(s-a,0)$, $(s+a,0)$ the vertices. By $\E^r_{c_1,c_2}$ we will denote the ellipse with foci at $(c_1,0)$ and $(c_2,0)$ and sum of the half-axes $r$; by $\widetilde{\E}^r_{a_1,a_2}$ we will denote the ellipse with vertices at $(a_1,0)$ and $(a_2,0)$ and sum of the half-axes  $r$.

\begin{lemma}\label{lemma:large_ellipse}
Let $a$ and set $r=a + \mathcal{C} \, \sqrt{a/2}$. Then 
\[
  \widetilde{\E}^r_{-a,a} \subseteq P = \{ x \in \R^2 : x_1 \geq -a + x_2^2/\mathcal{C}^2 \text{ and } x_1 \leq a - x_2^2 / \mathcal{C}^2 \} .
\]
\end{lemma}

\begin{proof}
  As both $P$ and $\widetilde{\E}^r_{-a,a}$ are symmetric with respect to the $x_2$-axis, it is sufficient to consider case $x_1 \in [-a,0]$ and the first inequality in the definition of $P$. For $x = (x_1,x_2) \in \widetilde{\E}^r_{-a,a}$ we have $x_2^2 \leq \mathcal{C}^2 \frac{ a^2 - x_1^2 }{ 2 a }$ and hence
  \[
    -a + \frac{x_2^2}{\mathcal{C}^2} \leq -a + \frac{ (a - x_1) \, (a + x_1) }{2a} \leq -a + a + x_1 = x_1 \, .
  \]
\end{proof}



In the next lemma, we prove that a certain family of smaller ellipses are contained in $\widetilde{\E}^{a+b}_{-a,a}$.

\begin{lemma}
  \label{lemma:small_ellipse}
  Let $a$, $b > 0$, $0 < \lambda \leq a$, $0 < \mu \leq \frac{b}{a} \lambda$ and $z \in [-a + \lambda , a - \lambda]$. Then $\widetilde{\E}^{\lambda + \mu}_{z - \lambda, z + \lambda} \subseteq \widetilde{\E}^{a+b}_{-a,a}$.
\end{lemma}

\begin{proof}
Let $(c +\lambda \cos\theta, \mu \sin\theta) \in \partial \widetilde{\E}^{\lambda + \mu}_{c-\lambda,c+\lambda}$ where $\theta \in [0,2\pi)$. As $c + \lambda \cos\theta \in [-a,a]$, there is a $\phi \in [0,\pi)$ such that 
\[
  c + \lambda \cos\theta = a \cos\phi \, .
\]
Thus $(a\cos\phi,\pm b\sin\phi) \in \partial \widetilde{\E}^{a+b}_{-a,a}$. We compare the squares of the $x_2$-coordinates.
\begin{align*}
    \mu^2 \sin^2 \theta
    & \leq \frac{ b^2 \lambda^2 }{ a^2 } \, \sin^2 \theta
    = \frac{ b^2 }{ a^2 } \left( \lambda^2 - \lambda^2 \, \cos^2 \theta \right) \\
    & = \frac{ b^2 }{a^2} \left( \lambda^2 - ( a \cos \phi - c )^2 \right)
    = b^2 \sin^2 \phi + \frac{ b^2 }{a^2} \left( \lambda^2 - a^2 + 2ac \, \cos \phi - c^2 \right) \\
    & = b^2 \sin^2 \phi + \frac{b^2}{a^2} \left( \lambda^2 - (a-c)^2 + 2ac \, (\cos \phi - 1 ) \right) .
\end{align*}
As $\lambda \leq a - c$, the second term is negative and we conclude $\mu^2 \sin^2 \theta \leq b^2 \sin^2 \phi$. We shown that the boundary of $\widetilde{\E}^{\lambda + \mu}_{c - \lambda, c + \lambda}$ is inside the boundary of $\widetilde{\E}^{a+b}_{-a,a}$ which proves the assertion.
\end{proof}

With the following corollary we apply the results of Lemma \ref{lemma:large_ellipse} and \ref{lemma:small_ellipse} to ellipses contained in the set $\mathbf{A}_m$.

\begin{corollary}\label{th:ellipse}
 Let $A_m = (a_0, a_1)$, $m = 1$, $2$, and set $\hat{z}=(a_0+a_1)/2$. Let $a \leq |A_m| / 2$ and set $b = \mathcal{C} \, \sqrt{a/2}$, $r = a + b$. Choose $\lambda$ and $\mu$ such that all assumptions of Lemma \ref{lemma:small_ellipse} are satisfied. Then $\widetilde{\E}^{\lambda+\mu}_{\hat{z} + z -\lambda, \hat{z} + z +\lambda} \subseteq \widetilde{E}^r_{\hat{z}-a,\hat{z}+a} \subseteq \overline{\mathbf{A}_m}$.
\end{corollary}


\section{Nonuniform meshes for the inverse Bloch transform}
\label{sec:meshes}

In this section, we introduce a method based on nonuniform meshes for numerical integration of functions of one variable with a square root singularity. Later on, we extend the method to approximate the inverse Bloch transform.

Let us start by considering the numerical approximation of the integral
\begin{equation}
 \label{eq:int_sample}
 I(\xi):=\int_0^{h} \xi(t) \, \d t \, ,
\end{equation}
where $\xi(t) = \xi_1(t) + \sqrt{t} \, \xi_2(t)$, $t \in [0,h]$, and both $\xi_1$ and $\xi_2$ are analytic in $[0,h]$.

Given a positive parameter $p \in (0,1)$ and $N \in \N$, the  method is described as follows. Let the nodal points be defined as
\[
  t_0=0, \qquad t_n = p^{n-1} h \, , \quad n = 1, \ldots, N+1 \, .
\]
These points are the end points of the subintervals
\[
  J_0 = [t_0,t_{N+1}], \qquad J_n = [t_{n+1}, t_n], \quad n=1,\dots,N \, .
\]
The integrand $\xi$ is analytic in any $J_n$, $n=1,\dots,N$, and we use an $M$-point Gauss–Legendre quadrature to approximate the integral on any such interval. On the interval $J_0$, the trapezoidal rule is used. Let the points and weights for the $M$-point Gauss–Legendre quadrature in $[-1,1]$ be denoted by
\[
  \left\{ (\tau_j, \, w_j) : \, j=1,\dots,M \right\} .
\]
For $n=1, \ldots, N$, the integral $I_n(\xi) := \int_{t_{n+1}}^{t_n} \xi(t) \, \d t$ is approximated by
\[
I_n^{M}(\xi)= \frac{p^{n-1}h - p^n h}{2} \sum_{j=1}^M \xi\left( \frac{p^{n-1}h-p^n h}{2} \, \tau_j + \frac{p^{n-1}h+p^n h}{2} \right) w_j.
\]
For $n=0$, the integral $I_0(\xi):=\int_{0}^{p^N h}\xi(t)\d t$ is approximated by
\[
 I_0^N(\xi)=\frac{p^N h}{2}\xi(0)+ \frac{p^N h}{2}\xi(p^N h).
\]
Thus, the complete composite quadrature formula is
\begin{multline}
 \label{eq:adaptive_sample}
 I_{N,M}(\xi) 
 = \sum_{n=1}^N I_n^{M}(\xi)+I_N^0(\xi)
 = \frac{p^N h}{2}\xi(0) + \frac{p^N h}{2} \xi(p^N h) \\
 {} + \sum_{n=1}^N\left[\frac{p^{n-1}h-p^n h}{2}\sum_{j=1}^M\xi\left(\frac{p^{n-1}h-p^n h}{2} \tau_j+\frac{p^{n-1}h+p^n h}{2}\right)w_j\right].
\end{multline}
Our goal is to estimate the error of the approximation of $I(\xi)$ by $I_{N,M}(\xi)$. We first quote a derivative free error estimate for Gaussian quadrature, for details we refer to \cite{Saute2007}. We also recall our notation for ellipses from before Lemma \ref{lemma:large_ellipse}.

\begin{theorem}[Theorem 5.3.13, \cite{Saute2007}]
  \label{th:err_Gaussian}
  Let $\rho > 1/2$ and consider the ellipse $\E_{0,1}^\rho$ as a subset of the complex plane. Let $\zeta : \, [0,1] \rightarrow \C$ be real analytic with complex analytic extension to $\E_{0,1}^\rho$. Denote by $I$ the integral over $(0,1)$ with integrand $\zeta$ and by $Q_M$ its approximation by the $M$-point Gauss-Legendre quadrature. Then
  \[
    | I - Q_M| \leq C \, (2\rho)^{-2M} \max_{z \in \partial \E^\rho_{0,1}} |\zeta(z)| \, .
  \]
\end{theorem}

The result can be extended to more general cases. Suppose $\zeta$ is an analytic function in $[\alpha, \beta]$ and let $\rho > \frac{\beta - \alpha}{2}$. If $\zeta$ can be analytically extended to $\E_{\alpha,\beta}^\rho$, then with analogous notation,
\begin{equation}
 \label{eq:err_Gaussian}
 | I - Q_M | \leq C \, \left( \frac{2\rho}{\beta - \alpha} \right)^{-2M} \max_{z \in \partial \E^\rho_{\alpha, \beta}} |\zeta(z)| \, .
\end{equation}
We apply these results to estimating the error in our quadrature rule for each interval $J_n$, $n=1,\dots,N$.

\begin{lemma}\label{th:err_I_n}
  Suppose $\xi$ can be analytically extended to an ellipse $\widetilde{\E}_{0,2h}^{\rho h}$ with vertices $0$, $2h$ and sum of semi-axis $\rho h$ where $1<\rho< 2$. For any $n=1,\dots,N$, there is a constant $C>0$ such that
  \begin{equation}
    \left|I_n(\xi)-I_n^{M}(\xi)\right|\leq C \left( \min\left\{ \frac{ 1 + \sqrt{p}  }{1 - \sqrt{p}}   \, , \; \sqrt{ \frac{\rho}{2 - \rho} } \right\}\right)^{-2M}\max_{z\in\partial \widetilde{\E}^{\rho h}_{0,2h}}|\xi(z)|.
  \end{equation}
\end{lemma}

\begin{proof}
  We wish to apply \eqref{eq:err_Gaussian} and thus need to find the largest ellipse with foci at $p^n h$ and $p^{n-1}h$ that lies inside of $\E^{\rho h}_{2h}$. Denote the semi-axis by $\lambda > \mu$, respectively and note that the linear eccentricity is $c = \left[\frac{p^{n-1}-p^n}{2}\right] h$. Hence, we have the necessary condition
  \[
    \mu^2 + \left[ \frac{p^{n-1}-p^n}{2} \right]^2 h^2 = \lambda^2 \, .
  \]
   We wish to apply Corollary \ref{th:ellipse} and in the notation there we have $\hat{z} = h$, $z = \left[\frac{p^{n-1}+p^n}{2}\right]h$, $a = h$ and $b = (\rho - 1) \, h$. The necessary conditions to apply the corollary hence are
  \[
    0 < \lambda \leq \left[\frac{p^{n-1}+p^n}{2}\right] h \, , \qquad \qquad
    \frac{\mu}{\lambda} \leq \rho - 1 \, .
  \]
  Our goal is to maximize $\lambda + \mu$ within these constraints. Note that for $\lambda > 0$, the line $\mu = (\rho - 1) \, \lambda$ intersects the hyperbola $\mu^2 + \left[ \frac{p^{n-1}-p^n}{2} \right]^2 h^2 = \lambda^2$ in exactly one point $(\hat{\lambda}, \hat{\mu})$, where
  \[
    \hat{\lambda} = \frac{ 1 - p }{2 \, \sqrt{2 \rho - \rho^2} } \, p^{n-1} h \, , \qquad
    \hat{\lambda} + \hat{\mu} = \sqrt{ \frac{ \rho }{ 2 - \rho} } \, \frac{1- p}{2} \, p^{n-1} h \, .
  \]
  If $(1+p)/2 \, p^{n-1} h < \hat{\lambda}$, we obtain the maximal value
  \[
    \lambda + \mu = \left( \frac{1+p}{2} + \sqrt{p} \right) p^{n-1} h = \frac{ ( 1 + \sqrt{p} )^2 }{ 2 } \, p^{n-1} h \, .
  \]
  Thus
  \[
    \lambda = \min \left\{  \frac{1 + p}{2} \, , \;  \frac{ 1 - p }{2 \, \sqrt{2 \rho - \rho^2} }\right\} p^{n-1} h
  \]
  and 
  \[
    \frac{ \lambda + \mu}{ c } = \min \left\{ \frac{ 1 + \sqrt{p}  }{1 - \sqrt{p}} \, , \; \sqrt{ \frac{\rho}{2 - \rho} } \right\} .
  \]
   Using \eqref{eq:err_Gaussian} and the maximum modulus principle in complex analysis, we obtain 
  \begin{align*}
    | I_n(\xi) - I^n_M(\xi) | 
   \leq C \left( \min \left\{ \frac{ 1 + \sqrt{p}  }{1 - \sqrt{p}}   \, , \; \sqrt{ \frac{\rho}{2 - \rho} } \right\} \right)^{-2M} \max_{z\in\partial \widetilde{\E}^{\rho h}_{0,2h}} \, |\xi(z)| \, .
 \end{align*}
The proof is finished.
\end{proof}

\noindent We also estimate the error of the trapezoidal rule on $J_0$.

\begin{lemma}\label{th:err_I_0}
 The error of the trapezoidal rule on $J_0$ is bounded by
 \begin{equation}
 \left|I_0(\xi)-I_0^N(\xi)\right|\leq C(p^N h)^{3/2}.
 \end{equation}
\end{lemma}

\begin{proof}
  Recall the representation $\xi(t) = \xi_1(t) + \sqrt{t} \, \xi_2(t)$. It is a standard result that the error in approximating the integral over the analytic function $\xi_1$ by the trapezoidal rule is of order $\mathrm{O}\left((p^N h)^2\right)$. Thus we only consider the second term, which we approximate by linear interpolation
  \[
    \xi_\text{lin}(t)=\frac{1}{\sqrt{p^N h}} \, \xi_2(p^N h) \, t \, .
  \]
  For any $t \in J_0 = [0, p^N h]$,
  \[
    \left| \sqrt{t} \, \xi_2(t) - \xi_\text{lin}(t) \right|
    = \left| \sqrt{t} \, \xi_2(t) - \frac{\xi_2(p^N h)}{\sqrt{p^N h}}t\right|
    \leq 2 \sqrt{p^N h} \, \sup_{t\in J_0} \left|\xi_2(t)\right|
    \leq C(p^N h)^{1/2}.
  \]
  and the application of the trapezoidal rule can be estimated by
  \[
   \left| \int_0^{p^N h} \left[\sqrt{t} \, \xi_2(t)  - \xi_\text{lin}(t)\right] \d t \right|
   \leq C(p^N h)^{3/2}.
  \]
\end{proof} 
  
With Lemma \ref{th:err_I_0} and \ref{th:err_I_n}, we are now prepared to  state an error estimate for the complete composite quadrature rule:

\begin{theorem}
  \label{theo:quad_est}
  When $N$ and $M$ are two positive integers, there is a constant $C>0$ such that
  \begin{equation}
    \left|I(\xi)-I_{N,M}(\xi)\right|
    \leq C N \left( \min \left \{ \frac{ 1 + \sqrt{p}  }{1 - \sqrt{p}}  \, , \; \sqrt{ \frac{\rho}{2 - \rho} } \right\} \right)^{-2M} + C(p^N h)^{3/2}.
  \end{equation}
\end{theorem}

\begin{proof}
 Combine Lemmas \ref{th:err_I_n} and \ref{th:err_I_0}.
\end{proof}


We conclude by apply the method introduced above to the approximation of the inverse Bloch transform,
\begin{equation}
(\J^{-1}w)(x)=\int_{-\underline{k}}^{1-\underline{k}} w(\alpha,x) \, \EE^{\i\alpha x_1}\d\alpha,\quad x\in\Omega^{2\pi}_H.
\end{equation}
Depending on the different cases in the definition of $E$, this equation is
\begin{equation}
\label{eq:inverse_bloch}
(\J^{-1}w)(x)=\begin{cases}
\displaystyle 
\int_{-\underline{k}}^{1/2-\underline{k}} w(\alpha,x) \, \EE^{\i\alpha x_1} \, \d\alpha + \int_{1/2-\underline{k}}^{1-\underline{k}} w(\alpha,x) \, \EE^{\i\alpha x_1} \, \d\alpha \, , \quad\text{ when }\underline{k}=0,\, \frac{1}{2};\\
\\
\displaystyle 
\begin{aligned}
  &\int_{-\underline{k}}^0 w(\alpha,x) \, \EE^{\i\alpha x_1} \, \d\alpha + \int_0^{\underline{k}} w(\alpha,x) \, \EE^{\i\alpha x_1} \, \d\alpha \\
  & \quad {} + \int_{\underline{k}}^{1/2} w(\alpha,x) \, \EE^{\i\alpha x_1} \, \d\alpha + \int_{1/2}^{1-\underline{k}} w(\alpha,x) \, \EE^{\i\alpha x_1} \, \d\alpha \, , \quad \text{otherwise.}
\end{aligned}
\end{cases}
\end{equation}
Note that in each interval, $w(\alpha,x)$ depends analytically on $\alpha$ except for a square root singularity at one edge point. From the definition of $\underline{k}$, the length of each interval is not larger than $1/2$. 
With a change of variables, we can rewrite any integral in the form
\[
  \int_0^h \phi(\alpha, x) \d\alpha \, , 
\]
where $\phi$ has the form 
\[
  \phi(\alpha,x) = \phi_1(\alpha,x) + \sqrt{\alpha} \, \phi_2(\alpha,x)
\]
with $\phi_1$, $\phi_2 \in C^\omega\left([0,2h];S(D)\right)$. From Theorem \ref{theo:A_m_2_parabolas} and Corollary \ref{th:ellipse} we know that $\phi$ can be extended analytically to $\widetilde{\E}^r_{0,2h}$ with $r = h + \mathcal{C} \, \sqrt{h/2}$. Thus $\rho$ in Lemma and Theorem can be chosen as
\[
  \rho = \min \left\{ \frac{r}{h} , \frac{3}{2} \right\} 
  = \min \left\{ 1 + \frac{\mathcal{C}}{\sqrt{2h}} , \frac{3}{2} \right\} .
\]

\noindent We redefine the integrals with new integrand as
\[
  I_0(\phi)(x) = \int_0^{p^N h} \phi(\alpha,x) \, \d\alpha, \qquad
  I_n(\phi)(x ) =\int_{p^n h}^{p^{n-1}h} \phi(\alpha,x) \, \d\alpha , \quad n=1,2,\dots,N.
\] 
The numerical approximations are
\begin{align*}
   I_N^0(\phi)(x) & = \frac{p^N h}{2} \, \phi(0,x) + \frac{p^N h}{2} \, \phi(p^N h,x) \, ; \\
   I^n_M(\phi)(x) & = \frac{p^{n-1}h-p^n h}{2} \sum_{j=1}^M \phi\left( \frac{p^{n-1}h - p^n h}{2} \tau_j + \frac{p^{n-1}h + p^n h}{2},x\right) w_j \, .
\end{align*}
We can now apply Theorem \ref{theo:quad_est} to the approximation of any of the integrals in \eqref{eq:inverse_bloch}. 

\begin{theorem}
  \label{th:err_inv_bloch}
  There exists constants $C > 0$ and $\Theta > 1$ such that 
  \begin{equation}
    \|I(\phi)-I_{N,M}(\phi)\|_{S(D)}\leq C \left( N \, \Theta^{-2M} + (p^N h)^{3/2} \right) .
   \end{equation}
\end{theorem}

\begin{proof}
  Set
  \[
    \Theta = \min\left\{ \frac{ 1 + \sqrt{p}  }{1 - \sqrt{p}}   \, , \; 
      \sqrt{ \frac{\sqrt{2h} + \mathcal{C}}{ \sqrt{2h} - \mathcal{C} } }  \, , \;
      \sqrt{3}
      \right\} > 1 \, .
  \]
  From our choice of $\rho$ and Theorem \ref{theo:quad_est}, the result follows.
\end{proof}

%



\section{Numerical approximation of scattering problems}

\subsection{Error estimation}
In this section, we conclude our analysis by providing error estimates for the numerical solution of the original scattering problem \eqref{eq:sca1}-\eqref{eq:sca3}. The algorithm can be divided into three steps:
\begin{algorithm}
\label{alg_num}
\begin{enumerate}
 \item Depending on $k$, find all the nodal points $\alpha_j$ and weights $\sigma_j$ where $j=1,2,\dots,L$.
 \item For any $\alpha_j$, compute the numerical solution of $w_\epsilon(\alpha_j,x)$, where $\epsilon>0$ is a parameter corresponding to the discretization (see below).
 \item Compute $u_{N,M,\epsilon}$ by the inverse Bloch transform:
 \[
 u_{N,M,\epsilon}(x):=\sum_{j=1}^L w_\epsilon\left(\alpha_j,x\right)e^{\i\alpha_j x_1}\sigma_j.
 \]
\end{enumerate}
\end{algorithm}

In this algorithm, it remains to discuss the second step, i.e. how to approximate $w(\alpha_j,x)$ numerically for any fixed $\alpha_j$. In principle, this may be carried out by any preferred numerical method for solving a boundary value problem in a periodic domain such as the integral equation method or the finite element method. In the present work, we have chosen the latter approach.

Assume that $\mathcal{M}_\epsilon$ is a family of regular, quasi-uniform triangular meshes in the finite domain $\Omega_H^{2\pi}$ with mesh width $0<\epsilon\leq \epsilon_0$, for some sufficiently small $\epsilon_0$. For simplicity, we assume that the nodal points on the left and right boundaries have got identical $x_2$-coordinates. We omit all the nodal points on the left boundary by imposing periodic boundary conditions at $-\pi$ and $\pi$, as well as  those on $\Gamma^{2\pi}$ due to the Dirichlet boundary conditions, and number the remaining nodes from $1$ to $M_0$. Let $\psi^j_M$, $j=1,2,\dots,M_0$, denote the globally continuous function that is $2\pi$-periodic with respect to $x_1$, linear on each mesh triangle and equals to $1$ at nodal point $j$ as well as to $0$ at all other nodal points. We define the space spanned by these functions by 
\[
 V_{\text{per},\epsilon} := {\rm span}\left\{\psi^j_{M_0}(x):\,j=1,2,\dots,M_0\right\}\subset \widetilde{H}^1_0(\Omega^{2\pi}_H) \, .
\]
For any fixed $\alpha$, we have the following error estimate for the Galerkin approximation to the solution of \eqref{eq:var_single}. For details we refer to \cite[Theorem 14]{Lechl2016a}.

\begin{theorem}
\label{th:err_FEM_alpha}
 Suppose that $\zeta\in C^{1,1}(\R)$. For any $\alpha\in W$, let $F(\alpha,\cdot)\in H^{1/2}_\p(\Gamma^{2\pi}_H)$ and denote by $w(\alpha,\cdot)$ the solution of the variational equation \eqref{eq:var_single}. Then $w(\alpha,\cdot) \in H^2(\Omega^{2\pi}_H)$. Moreover, when $\epsilon_0>0$ is sufficiently small and $w_\epsilon(\alpha,\cdot) \in V_{\p,\epsilon}$ solves
 \begin{equation}
  \label{eq:FEM_alpha}
  a_\alpha(w_\epsilon(\alpha,\cdot),\phi_\epsilon)=\int_{\Gamma^{2\pi}_H}F(\alpha,\cdot) \, \overline{\phi_\epsilon} \, \d s\quad\text{ for all }\phi_\epsilon\in V_{\p,\epsilon},
 \end{equation}
then 
\[
 \|w_\epsilon(\alpha,\cdot)-w(\alpha,\cdot)\|_{H^\ell(\Omega_H^{2\pi})}\leq C \epsilon^{2-\ell}\|F(\alpha,\cdot)\|_{H^{1/2}_\p(\Gamma^{2\pi}_H)},\quad \ell=0,1,
\]
where $C$ is independent of $\alpha\in W$.
\end{theorem}

From Theorems \ref{th:err_inv_bloch} and \ref{th:err_FEM_alpha}, we can immediately derive an error estimate for the solution computed using Algorithm \ref{alg_num}.

\begin{theorem}
 \label{th:err_num}
 Suppose that $f\in H^{1/2}_r(\Omega_H)$ with $r\in (1/2,1)$ such that $F(\alpha,x):=(\J f)(\alpha,x) \in \mathcal{A}^\omega\left((-\underline{k},1-\underline{k}];H^{1/2}_\p(\Gamma^{2\pi}_H);E\right)$. Then the error between numerical approximation $u_{N,M,\epsilon}$ from Algorithm \ref{alg_num} and the exact solution $u$ is bounded by
 \[
  \|u_{N,M,\epsilon}-u\|_{H^\ell(\Omega^{2\pi}_H)}\leq C\left[\epsilon^{2-\ell} + N \Theta^{-2M} + p^{3N/2}\right],\quad \ell=0,1,
 \]
where $\Theta>1$ is defined as in Theorem \ref{th:err_inv_bloch} and $C$ depends on $\|f\|_{H^{1/2}_r(\Omega_H)}$, $k$ and $p$.

\end{theorem}

\begin{proof}
 We only present detailed arguments for the case that $\underline{k}=0, 0.5$. For the other cases, the proof is carried out similarly. Using the estimate from Theorem \ref{th:err_inv_bloch}, we obtain
\[
 \begin{aligned}
   \|u_{N,M,\epsilon}-u\|_{H^\ell(\Omega^{2\pi}_H)}&\leq  \left\|\sum_{\ell=1}^2\sum_{m=1}^{NM+2}w_\epsilon\left(\alpha_m^\ell,\cdot\right)e^{\i\alpha_m^\ell()_1}w_m^\ell-\sum_{\ell=1}^2\sum_{m=1}^{NM+2}w\left(\alpha_m^\ell,\cdot\right)e^{\i\alpha_m^\ell()_1}w_m^\ell\right\|_{H^\ell(\Omega^{2\pi}_H)}\\
   &\qquad+\left\|\sum_{\ell=1}^2\sum_{m=1}^{NM+2}w\left(\alpha_m^\ell,\cdot\right)e^{\i\alpha_m^\ell()_1}w_m^\ell-u\right\|_{H^\ell(\Omega^{2\pi}_H)}\\
   &\leq \sum_{\ell=1}^2\sum_{m=1}^{NM+2}w_m^\ell\left\|w_\epsilon\left(\alpha_m^\ell,\cdot\right)-w\left(\alpha_m^\ell,\cdot\right)\right\|_{H^\ell(\Omega^{2\pi}_H)}+C \, N \, \Theta^{-2M}+Cp^{3N/2}\\
   &\leq C\epsilon^{2-\ell}\sum_{\ell=1}^2\sum_{m=1}^{NM+2}w_m^\ell\|F(\alpha_m^\ell,\cdot)\|_{H^{1/2}_\p(\Gamma^{2\pi}_H)} + C\, N \, \Theta^{-2M}+Cp^{3N/2},
 \end{aligned}
\]
where $()_1$ denotes the first coordinate of a two dimensional argument vector.  As $f\in H^{1/2}_r(\Gamma^{2\pi}_H)$ for $r>1/2$, by Theorem \ref{th:reg_bloch} we have $F=\J f\in  H^r((-\underline{k},1-\underline{k}];H^{1/2}_\p(\Gamma^{2\pi}_H))$, and 
 \[
  \|F\|_{H_0^r((-\underline{k},1-\underline{k}];H^{1/2}_\p(\Gamma^{2\pi}_H))}=\|f\|_{H^{1/2}_r(\Gamma^{2\pi}_H)}.
 \]
From Sobolev's embedding theorem, $F \in C^0((-\underline{k},1-\underline{k}];H^{1/2}_\p(\Gamma^{2\pi}_H))$ and
\[
 \|F\|_{C^0((-\underline{k},1-\underline{k}];H^{1/2}_\p(\Gamma^{2\pi}_H))}\leq C\|F\|_{H_0^r((-\underline{k},1-\underline{k}];H^{1/2}_\p(\Gamma^{2\pi}_H))}=\|f\|_{H^{1/2}_r(\Gamma^{2\pi}_H)}.
\]
From the fact that $\sum_{\ell=1}^2\sum_{m=1}^{NM+2}w_m^\ell=h$,
\[
\|u_{N,M,\epsilon}-u\|_{H^\ell(\Omega^{2\pi}_H)}\leq C\epsilon^{2-\ell}\|f\|_{H^{1/2}_r(\Gamma^{2\pi}_H)}+C\, N \, \Theta^{-2M}+Cp^{3N/2}\leq C\left[\epsilon^{2-\ell}+ N \, \Theta^{-2M}+p^{3N/2}\right].
\]
 The proof is finished.
\end{proof}

\subsection{Numerical experiments}

We present six numerical examples that demonstrate the convergence properties of  Algorithm \ref{alg_num}. In all of the examples, we use the same periodic surface given by
\[
 \zeta(t)=\frac{3}{2}+\frac{\sin t}{3}-\frac{\cos 2t}{4}.
\]
The following parameters are also fixed:
\[
 H=3 \, , \quad p=0.5 \, , \quad h=0.5 \, .
\]
We use two different wave numbers, $k=\sqrt{2}$ and $k = 10$, respectively. Note that when $k=\sqrt{2}$, $E=\{1-\sqrt{2},\sqrt{2}-1,2-\sqrt{2}\}$ whereas when $k=10$, $E=\{0,1\}$. 

We also use three different incident fields,
\[
    u^i_1(k,x)=\Phi(x,a_1)-\Phi(x,a'_1);\quad 
     u^i_2(k,x)=\Phi(x,a_2)-\Phi(x,a'_2);\quad
     u^i_3(k,x)=\int_{-\pi/2}^{\pi/2}e^{\i k x_1\sin t-\i k x_2\cos t}g(t)\d t.
\]
Here, $\Phi(x,y) = \frac{\i}{4}H_0^{(1)}(k|x-y|)$ denotes the free space fundamental solution of the Helmholtz equation and $H_0^{(1)}(\cdot)$ is the Hankel function of the first kind of order $0$. As source points we use $a_1=(0.4,0.2)^\bot$ and $a'_1=(0.4,-0.2)^\bot$;  $a_2=(0.4, 3)^\bot$ and $a'_2=(0.4,-3)^\bot$. $u^i_3$ is a downward propagating Herglotz wave function with the density function $g$ defined as
\[
g(t)=\begin{cases}\displaystyle
\frac{(x-a)^6(x-b)^6}{((b-a)/2)^{12}},\quad a<x<b;\\
\displaystyle
0,\quad\text{ otherwise;}
\end{cases}
\]
with $a=0.4,\,b=0.5$. 

\vspace{0.2cm}
With the definitions of the three incident fields, we apply Algorithm \ref{alg_num} to the following examples.
\noindent
\begin{itemize}
    \item {\bf Example 1. }$k=\sqrt{2}$, $u^i(x)=u^i_1\left(\sqrt{2},x\right)$.
\item {\bf Example 2. }$k=10$, $u^i(x)=u^i_1\left(10,x\right)$.
\item {\bf Example 3. }$k=\sqrt{2}$, $u^i(x)=u^i_2\left(\sqrt{2},x\right)$.
\item {\bf Example 4. }$k=10$, $u^i(x)=u^i_2\left(10,x\right)$.
\item {\bf Example 5. }$k=\sqrt{2}$, $u^i=u^i_3\left(\sqrt{2},x\right)$.
\item {\bf Example 6. }$k=10$, $u^i(x)=u^i_3\left(10,x\right)$.
\end{itemize}

For all the examples, we collect the value of $u_{N,M,\epsilon}$ on the line segment $\Gamma_h:=[-\pi,\pi]\times\{2.9\}$ and study the dependence of errors on the parameters $N$, $M$ and $\epsilon$. Supposing we know the exact solution $u_{exa}$ on $\Gamma_h$, we can compute the relative error defined by
\[
err_{N,M,\epsilon}:=\frac{\left\|u_{N,M,\epsilon}-u_{exa}\right\|_{L^2(\Gamma_h)}}{\left\|u_{exa}\right\|_{L^2(\Gamma_h)}}.
\]

In Examples 1 and 2, since $u^i$ is the half-space Green's function with source $(0.4,0.2)$ which lies below the periodic surface, $u^i$ satisfies the radiation condition \eqref{eq:sar}, i.e., $f=0$ in \eqref{eq:sca3}. Thus we have to modify the problem \eqref{eq:var}. In this case, we are looking for a solution $u^s\in H^1_r(\Omega_H)$ such that
\[
\int_{\Omega_H}\left[\nabla u^s\cdot\nabla\overline{\phi}-k^2 u^s\overline{\phi}\right]\d x-\int_{\Gamma_H}T^+\left(u\big|_{\Gamma_H}\right)\overline{\phi}\d s=0
\]
with the boundary condition $u^s=-u^i$ on $\Gamma$. It is well known that the exact solution $u_{exa}=-u^i$ in  $\Omega$.  For each example, we first fix sufficiently large $M$ and $N$ ($M=10$, $N=20$) and check the dependence of error on the finite element discretization, i.e., for $\epsilon=0.04$, $0.02$, $0.01$, $0.005$, we compute the relative errors. The results are shown in Table \ref{tab:group_1_1} and are plotted on both logarithm scales in Figure \ref{fig:egs} picture (a). Since the slopes of both curves are approximately $2$, the convergence rates coincide with that shown in Theorem \ref{th:err_num}.

Next, we study the convergence of Algorithm \ref{alg_num} with respect to the parameters $M$ and $N$. Fix $\epsilon=0.005$, and compute the relative errors with $M=2,3,4,5$ and $N=4,8,12,16,20$ (also $24$ when $k=\sqrt{2}$). The results are presented in Tables \ref{tab:group_1_2_1} and \ref{tab:group_1_2_2}. In both tables, showing results for Examples 1 and 2, respectively, we clearly observe convergence of the numerical solution when $N$ and $M$ increase. When both these numbers are sufficiently large, the errors no longer decay, since the discretization error of the finite element method plays the more important role. The errors appear to be more sensitive with respect to the parameter $N$, since for $M\geq 3$ we observe that the errors almost only depend on $N$.

\begin{table}[]
    \centering
    \begin{tabular}{|c|c|c|c|c|}
    \hline
    $k$ & $\epsilon=0.04$& $\epsilon=0.02$& $\epsilon=0.01$& $\epsilon=0.005$\\
    \hline\hline
        $\sqrt{2}$ & $3.5\times 10^{-4}$ &$8.8\times 10^{-5}$&  $2.2\times 10^{-5}$ & $6.2\times 10^{-6}$\\
        \hline
        $10$ & $8.5\times 10^{-2}$ & $2.2\times 10^{-2}$  &$5.5\times 10^{-3}$ & $1.4\times 10^{-3}$\\
        \hline
    \end{tabular}
    \caption{Relative errors of Example 1 and 2, with respect to $\epsilon$.}
    \label{tab:group_1_1}
\end{table}

\begin{table}[]
    \centering
    \begin{tabular}{|c|c|c|c|c|c|c|}
    \hline
     & $N=4$ & $N=8$ & $N=12$  & $N=16$& $N=20$ & $N=24$ \\
    \hline\hline
        $M=2$  &  $6.7\times10^{-2}$  &$4.3\times10^{-3}$ &$3.2\times10^{-4}$ & $1.3\times10^{-4}$& $1.3\times10^{-4}$ &$1.3\times10^{-4}$\\
        \hline
         $M=3$  &  $6.7\times10^{-2}$  &$4.3\times10^{-3}$ &$2.7\times10^{-4}$ & $2.0\times10^{-5}$&$6.2\times10^{-6}$ &$5.7\times10^{-6}$\\
            \hline
        $M=4$  &  $6.7\times10^{-2}$  &$4.3\times10^{-3}$ &$2.7\times10^{-4}$ &$2.0\times10^{-5}$ &$6.2\times10^{-6}$& $5.6\times10^{-6}$\\
        \hline
        $M=5$ &  $6.7\times10^{-2}$ &$4.3\times10^{-3}$ &$2.7\times10^{-4}$ &$2.0\times10^{-5}$ &$6.2\times10^{-6}$&$5.6\times10^{-6}$\\
         \hline
    \end{tabular}
    \caption{Relative errors of Example 1, with respect to $M$ and $N$.}
    \label{tab:group_1_2_1}
\end{table}

\begin{table}[]
    \centering
    \begin{tabular}{|c|c|c|c|c|c|}
    \hline
     & $N=4$ & $N=8$ & $N=12$  & $N=16$& $N=20$ \\
    \hline\hline
        $M=2$  &  $8.3\times 10^{-2}$  & $8.3\times 10^{-2}$&$8.3\times 10^{-2}$ & $8.3\times 10^{-2}$& $8.3\times 10^{-2}$ \\
        \hline
         $M=3$  & $5.6\times10^{-3}$   &$5.9\times10^{-3}$ & $5.9\times10^{-3}$&$5.9\times10^{-3}$& $5.9\times10^{-3}$\\
            \hline
        $M=4$  &  $1.9\times10^{-3}$ & $1.5\times10^{-3}$& $1.5\times10^{-3}$& $1.5\times10^{-3}$& $1.5\times10^{-3}$\\
        \hline
        $M=5$ &  $1.9\times10^{-3}$ & $1.4\times10^{-3}$& $1.4\times10^{-3}$& $1.4\times10^{-3}$&$1.4\times10^{-3}$\\
         \hline
    \end{tabular}
    \caption{Relative errors of Example 2, with respect to $M$ and $N$.}
    \label{tab:group_1_2_2}
\end{table}

In Example 3 and 4, the incident fields are point sources located above the periodic surface; and in Example 5 and 6, the incident fields are Herglotz wave functions propagating downwards. For all these examples, we no longer know the exact solutions. Instead, we choose parameters ($M=10$, $N=20$ and $\epsilon=0.005$) for which we expect the result to be sufficiently accurate  and use the corresponding numerical solutions $u_{N,M,\epsilon}$ as a reference solution instead of an exact solution. For a plot of the wave field in Example 6 we refer to Figure \ref{fig:eg6}.  For all examples, we first fix $N=20$ and compute relative errors with $M=2$, $3$, $4$, $5$; in a second set of computations, we fix $M=10$ and compute relative errors with $N=4$, $6$, $8$, $10$, $12$. The relative errors with respect to $M$ are given in Table \ref{tab:group_2_1} and plotted in (b) Figure \ref{fig:egs}, while relative errors with respect to $N$ are given in Table \ref{tab:group_2_2} and plotted (c) Figure \ref{fig:egs}. From both graphs, we clearly observe the exponential convergence we expected from Theorem \ref{th:err_num}.

At the end, we also discuss the convergence rates with respect to the parameters $M$ and $N$. First let's focus on the dependence of $M$. The slopes of the curves in (b) Figure \ref{fig:egs} are approximately $-4.4$ (Example 3), $-3.2$ (Example 4), $-2.1$ (Example 5) and $-3.4$ (Example 6).  Thus exponential convergence is observed with respect to the parameter $M$. Similarly,  slopes of the curves in (c) Figure \ref{fig:egs} are approximately $-0.68$ (Example 3), $-0.55$ (Example 4), $-0.79$ (Example 5) and $-0.70$ (Example 6) which show the exponential convergence with respect to the parameter $N$. The convergence results coincide with the theoretical results in Theorem \ref{th:err_num}.

\begin{table}[]
    \centering
    \begin{tabular}{|c|c|c|c|c|}
    \hline
     & $M=2$& $M=3$& $M=4$& $M=5$ \\
    \hline\hline
        Eg 3 & $7.4\times 10^{-4}$ & $9.0\times 10^{-6}$ & $1.0\times 10^{-7}$ & $1.6\times 10^{-9}$ \\
        \hline
         Eg 4 & $3.5\times 10^{-2}$  & $2.5\times 10^{-3}$& $1.0\times 10^{-4}$ & $2.5\times 10^{-6}$ \\
        \hline
        Eg 5 &$2.8\times 10^{-4} $  &$2.6\times 10^{-5} $ & $1.4\times 10^{-6} $ & $2.2\times 10^{-7} $\\
        \hline
         Eg 6 & $1.6\times 10^{-5} $ & $1.6\times 10^{-7} $&  $2.8\times 10^{-9} $& $5.9\times 10^{-11} $\\
        \hline
    \end{tabular}
    \caption{Relative errors of Example 3-6 with respect to $M$.}
    \label{tab:group_2_1}
\end{table}

\begin{table}[]
    \centering
    \begin{tabular}{|c|c|c|c|c|c|}
    \hline
     & $N=4$& $N=6$ & $N=8$& $N=10$ & $N=12$\\
    \hline\hline
        Eg 3 & $9.8\times 10^{-2}$& $2.6\times 10^{-2}$ & $6.8\times 10^{-3}$ &$1.7\times 10^{-3}$& $4.3\times 10^{-4}$\\
        \hline
         Eg 4  &$1.5\times 10^{-1}$& $7.5\times 10^{-2}$ &$2.6\times 10^{-2}$ & $7.6\times 10^{-3}$ & $2.0\times 10^{-3}$\\
        \hline
        Eg 5 &   $9.1\times10^{-2}$ & $2.3\times 10^{-2}$ & $5.7\times 10^{-3}$ &$1.4\times 10^{-3}$ & $3.8\times10^{-4}$\\
        \hline
         Eg 6 & $6.3\times 10^{-2}$ & $1.6\times 10^{-2}$& $4.0\times 10^{-3}$ &$9.9\times 10^{-4}$  &$2.5\times 10^{-4}$\\
        \hline
    \end{tabular}
    \caption{Relative errors of Example 3-6 with respect to $N$.}
    \label{tab:group_2_2}
\end{table}

\begin{table}[]
    \centering
    \begin{tabular}{c c c}
      (a)  & (b) & (c)\\
       \includegraphics[width=0.33\textwidth]{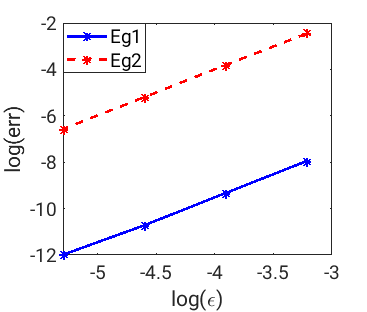}
         & 
          \includegraphics[width=0.33\textwidth]{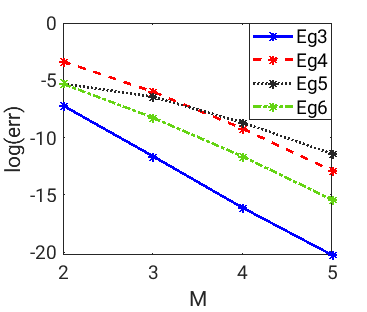}&
           \includegraphics[width=0.33\textwidth]{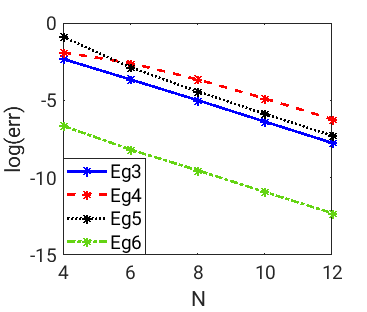}
    \end{tabular}
    \caption{Dependence of relative errors on parameters (a) $\epsilon$, (b) $M$ and (c) $N$.}
    \label{fig:egs}
\end{table}

\begin{table}[]
    \centering
    \begin{tabular}{c c }
      (a)  & (b)\\
       \includegraphics[width=0.45\textwidth]{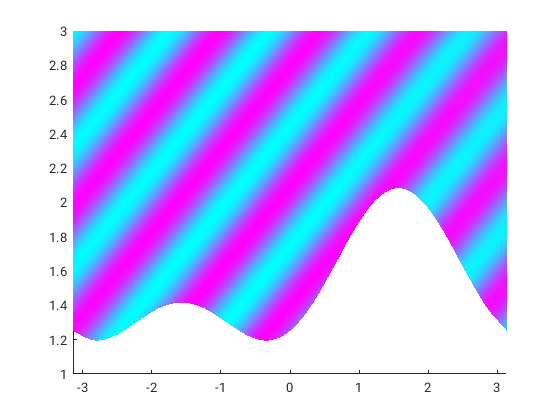}
         & 
          \includegraphics[width=0.45\textwidth]{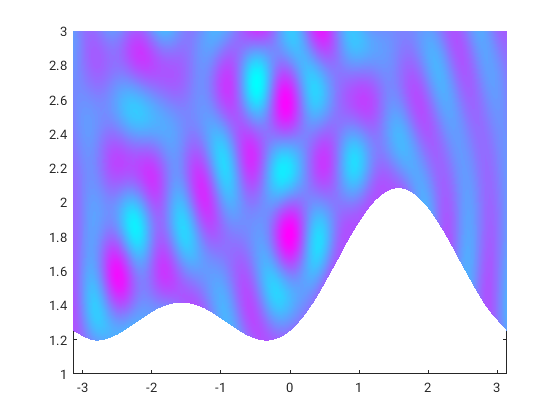}
    \end{tabular}
    \caption{Real part of waves in Example 6: (a) incident field, (b) scattered field.}
    \label{fig:eg6}
\end{table}

\section*{Acknowledgements.} 

This research was funded by the Deutsche Forschungsgemeinschaft (DFG, German Research Foundation) -- Project-ID 258734477 -- SFB 1173.

\bibliographystyle{plain}
\bibliography{ip-biblio} 

\providecommand{\noopsort}[1]{}
\begin{thebibliography}{10}

\bibitem{Abbou1993}
T.~Abboud.
\newblock Electromagnetic waves in periodic media.
\newblock In {\em Second International Conference on Mathematical and Numerical
  Aspects of Wave Propagation}, pages 1--9, Newark, DE, 1993. SIAM,
  Philadelphia.

\bibitem{Arens1999}
T.~Arens.
\newblock The scattering of plane elastic waves by a one-dimensional periodic
  surface.
\newblock {\em Math. Meth. Appl. Sci.}, 22:55--72, 1999.

\bibitem{Arens2006a}
T.~Arens, S.~N. Chandler-Wilde, and J.~A. DeSanto.
\newblock On integral equation and least squares methods for scattering by
  diffraction gratings.
\newblock {\em Communications in Computational Physics}, 1:1010--1042, 2006.

\bibitem{Arens2010}
Tilo Arens.
\newblock Scattering by biperiodic layered media: The integral equation
  approach, 2010.
\newblock Habilitation Thesis, Universit{\"a}t Karlsruhe.

\bibitem{Bao1994a}
G.~Bao.
\newblock Diffractive optics in periodic structures: the {T}{M} polarization.
\newblock Technical report, Institute for Mathematics and Its Applications,
  University of Minnesota, Minneapolis, 1994.

\bibitem{Bao1995}
G.~Bao.
\newblock Finite element approximation of time harmonic waves in periodic
  structures.
\newblock {\em {SIAM} Journal on Numerical Analysis}, 32(4):1155--1169, 1995.

\bibitem{Bao1996}
G.~Bao.
\newblock Numerical analysis of diffraction by periodic structures: {TM}
  polarization.
\newblock {\em Numer. Math.}, 75:1--16, 1996.

\bibitem{Bao1997}
G.~Bao.
\newblock Variational approximation of {M}axwell's equations in biperiodic
  structures.
\newblock {\em SIAM J. Appl. Math.}, 57:364--381, 1997.

\bibitem{Bao2000}
G.~Bao and D.~C. Dobson.
\newblock On the scattering by a biperiodic structure.
\newblock {\em Proc. Amer. Math. Soc.}, 128:2715--2723, 2000.

\bibitem{Bruno1993}
O.~Bruno and F.~Reitich.
\newblock Numerical solution of diffraction problems: A method of variation of
  boundaries {III}. {D}oubly-periodic gratings.
\newblock {\em J. Opt. Soc. Amer.}, 10:2551--2562, 1993.

\bibitem{Chand2010}
S.~N. {Chandler-Wilde} and J.~Elschner.
\newblock Variational approach in weighted {S}obolev spaces to scattering by
  unbounded rough surfaces.
\newblock {\em SIAM. J. Math. Anal.}, 42:2554--2580, 2010.

\bibitem{Chand2005}
S.~N. {Chandler-Wilde} and P.~Monk.
\newblock Existence, uniqueness, and variational methods for scattering by
  unbounded rough surfaces.
\newblock {\em SIAM. J. Math. Anal.}, 37:598--618, 2005.

\bibitem{Chand2002}
S.~N. Chandler-Wilde, M.~Rahman, and C.~R. Ross.
\newblock A fast two-grid and finite section method for a class of integral
  equations on the real line with application to an acoustic scattering problem
  in the half-plane.
\newblock {\em Numer. Math.}, 93:1--51, 2002.

\bibitem{Coatl2012}
J.~Coatl{\'e}ven.
\newblock {Helmholtz equation in periodic media with a line defect}.
\newblock {\em {J. Comp. Phys.}}, 231:1675--1704, 2012.

\bibitem{Dobso1994}
D.~C. Dobson.
\newblock A variational method for electromagnetic diffraction in biperiodic
  structures.
\newblock {\em Math. Model. Numer. Anal.}, 28:419--439, 1994.

\bibitem{Dobso1992a}
D.~C. Dobson and A.~Friedman.
\newblock The time-harmonic {M}axwell equations in biperiodic structures.
\newblock {\em Math. Anal. Appl.}, 166:507--528, 1992.

\bibitem{Elsch1998}
J.~Elschner and G.~Schmidt.
\newblock Diffraction of periodic structures and optimal design problems of
  binary gratings. {P}art {I}: Direct problems and gradient formulas.
\newblock {\em Math. Meth. Appl. Sci.}, 21:1297--1342, 1998.

\bibitem{Hadda2015}
H.~Haddar and T.~P. Nguyen.
\newblock {Volume integral method for solving scattering problems from locally
  perturbed periodic layers}.
\newblock In {\em WAVES 2015 Proceed.}, KIT, Karlsruhe, 2015.

\bibitem{Hasel2004}
K.~Haseloh.
\newblock {\em Second Kind Integral Equations on the Real Line: Solvability and
  Numerical Analysis in Weighted Spaces}.
\newblock PhD thesis, Universit{\"a}t Hannover, 2004.

\bibitem{Kirsc1993}
A.~Kirsch.
\newblock Diffraction by periodic structures.
\newblock In L.~P{\"a}varinta and E.~Somersalo, editors, {\em Proc. Lapland
  Conf. on Inverse Problems}, pages 87--102. Springer, 1993.

\bibitem{Lechl2008a}
A.~Lechleiter.
\newblock {\em Factorization Methods for Photonics and Rough Surface
  Scattering}.
\newblock PhD thesis, Universit{\"a}t Karlsruhe, Karlsruhe, Germany, 2008.

\bibitem{Lechl2016}
A.~Lechleiter.
\newblock The {F}loquet-{B}loch transform and scattering from locally perturbed
  periodic surfaces.
\newblock {\em J. Math. Anal. Appl.}, 446(1):605--627, 2017.

\bibitem{Lechl2015e}
A.~Lechleiter and D.-L. Nguyen.
\newblock {Scattering of {H}erglotz waves from periodic structures and mapping
  properties of the {B}loch transform}.
\newblock {\em {Proc. Roy. Soc. Edinburgh Sect. A}}, 231:1283--1311, 2015.

\bibitem{Lechl2016a}
A.~Lechleiter and R.~Zhang.
\newblock A convergent numerical scheme for scattering of aperiodic waves from
  periodic surfaces based on the {F}loquet-{B}loch transform.
\newblock {\em SIAM J. Numer. Anal}, 55(2):713--736, 2017.

\bibitem{Lechl2017}
A.~Lechleiter and R.~Zhang.
\newblock A {F}loquet-{B}loch transform based numerical method for scattering
  from locally perturbed periodic surfaces.
\newblock {\em SIAM J. Sci. Comput.}, 39(5):B819--B839, 2017.

\bibitem{Lechl2016b}
A.~Lechleiter and R.~Zhang.
\newblock Non-periodic acoustic and electromagnetic scattering from periodic
  structures in 3d.
\newblock {\em Comput. Math. Appl.}, 74(11):2723--2738, 2017.

\bibitem{Li2016}
J.~Li, G.~Sun, and R.~Zhang.
\newblock The numerical solution of scattering by infinite rough surfaces based
  on the integral equation method.
\newblock {\em Comput. Math. Appl.}, 71(7):1491--1502, 2016.

\bibitem{Li2011}
P.~Li, H.~Wu, and W.~Zheng.
\newblock Electromagnetic scattering by unbounded rough surfaces.
\newblock {\em SIAM J. Math. Anal.}, 43(3):1205--1231, 2011.

\bibitem{Meier2000}
A.~Meier, T.~Arens, S.~N. {Chandler-Wilde}, and A.~Kirsch.
\newblock A {N}ystr{\"o}m method for a class of integral equations on the real
  line with applications to scattering by diffraction gratings and rough
  surfaces.
\newblock {\em J. Int. Equ. Appl.}, 12:281--321, 2000.

\bibitem{Meier2001}
A.~Meier and S.~N. {Chandler-Wilde}.
\newblock On the stability and convergence of the finite section method for
  integral equation formulations of rough surface scattering.
\newblock {\em Math. Methods. Appl. Sci.}, 24:209--232, 2001.

\bibitem{Nedel1991}
J.-C. N{\'e}d{\'e}lec and F.~Starling.
\newblock Integral equation methods in a quasi-periodic diffraction problem for
  the time-harmonic {M}axwell’s equations.
\newblock {\em SIAM J. Math. Anal.}, 22(6):1679--1701, 1991.

\bibitem{Saute2007}
S.~Sauter and C.~Schwab.
\newblock {\em Boundary Element Methods}.
\newblock Springer, Berlin-New York, 2007.

\bibitem{Schmi2003}
G.~Schmidt.
\newblock On the diffraction by biperiodic anisotropic structures.
\newblock {\em Appl. Anal.}, 82:75--92, 2003.

\bibitem{Stryc1998}
B.~Strycharz.
\newblock An acoustic scattering problem for periodic, inhomogeneous media.
\newblock {\em Math. Method Appl. Sci.}, 21(10):969--983, 1998.

\bibitem{Zhang2017e}
R.~Zhang.
\newblock A high order numerical method for scattering from locally perturbed
  periodic surfaces.
\newblock {\em SIAM J. Sci. Comput.}, 40(4):A2286--A2314, 2018.

\end{thebibliography}

\end{document}